\documentclass[11pt]{elsarticle}
\usepackage[text={6in,9in},centering]{geometry}
\usepackage{graphicx,color,psfrag}
\usepackage{amsmath,amsthm,amssymb,mathrsfs,upgreek,cancel,enumitem}
\usepackage[colorlinks=false]{hyperref}
\newcommand{\N}{\ensuremath{\mathbb{N}}}

\newcommand{\R}{\ensuremath{\mathbb{R}}}

\newcommand{\e}{e}
\newcommand{\diver}{\mathop{\mathrm{div}}}

\newcommand{\dotdot}{\mathbin :}
\newcommand{\norm}[1]{\lVert #1 \rVert}
\newcommand{\V}[1]{\boldsymbol{#1}}
\newcommand{\T}[1]{\boldsymbol{#1}}
\renewcommand{\d}{d}

\renewcommand{\L}{\boldsymbol{\mathsf L}}
\newcommand{\Ell}{\mathscr{L}}
\newtheorem{theorem}{Theorem}[section]
\newtheorem{definition}[theorem]{Definition}
\newtheorem{lemma}[theorem]{Lemma}
\newtheorem{proposition}[theorem]{Proposition}
\newtheorem{remark}[theorem]{Remark}

\title{On the free rotations of rigid bodies with a liquid-filled gap}
\author{Giusy Mazzone}
\date{}
\address{\small Department of Mathematics and Statistics
\\
Queen's University
\\
Kingston, ON K7L 3N6}
\ead{giusy.mazzone@queensu.ca}

\makeatletter
\def\ps@pprintTitle{%
   \let\@oddhead\@empty
   \let\@evenhead\@empty
   \def\@oddfoot{\reset@font\hfil\thepage\hfil}
   \let\@evenfoot\@oddfoot
}
\makeatother

\begin{document}
\begin{frontmatter}
%
%
\begin{abstract}
We consider the system constituted by a hollow rigid body whose cavity contains a homogeneous rigid ball, and let the gap between the solids be entirely filled by a viscous incompressible fluid. We investigate the free rotations of the whole system, i.e., motions driven only by the inertia of the fluid-solids system once an initial angular momentum is imparted on the whole system. We prove the existence of global weak solutions and local strong solutions to the equations of motion. In addition, we prove that the fluid velocity as well as the inner core angular velocity relative to the outer solid converge to zero as time approaches infinity. 
\end{abstract}

\end{frontmatter}
\section{Introduction}
Consider the system constituted by a hollow rigid body $\mathcal B_1$ whose cavity contains a homogeneous rigid ball $\mathcal B_2$. Let the gap between $\mathcal B_1$ and $\mathcal B_2$ be entirely filled by a viscous incompressible fluid $\Ell$ (simply called {\em liquid}). Let $G$ be the center of mass of the system $\mathscr S_C$ constituted by the outer rigid body $\mathcal{B}_1$ and the liquid. Suppose that $G$ is a fixed point in space and time with respect to an inertial frame  of reference $\mathcal I$, and it coincides with the (geometrical) center of the ball $\mathcal B_2$. \footnote{ The geometrical center of the ball is also its center of mass due to the homogeneity and geometrical symmetry of $\mathcal B_2$. }  We are interested in the {\em free rotations} of the whole system of {\em rigid bodies with a  liquid-filled gap}. This type of motion occurs when no external forces and torques are applied, and the system is constrained to rotate (without friction) around $G$ driven by only its inertia once an initial angular momentum is imparted, see Figure \ref{fig:liquid_gap}. 
\begin{figure}[h]
\centering
\psfrag{1}{$\mathcal B_2$}
\psfrag{2}{$\Ell$}
\psfrag{3}{$\mathcal B_1$}
\psfrag{G}{\footnotesize $G$}
\psfrag{o}{$\V \omega_{20}$}
\psfrag{a}{$\V \omega_{10}$}
\includegraphics[width=.5\textwidth]{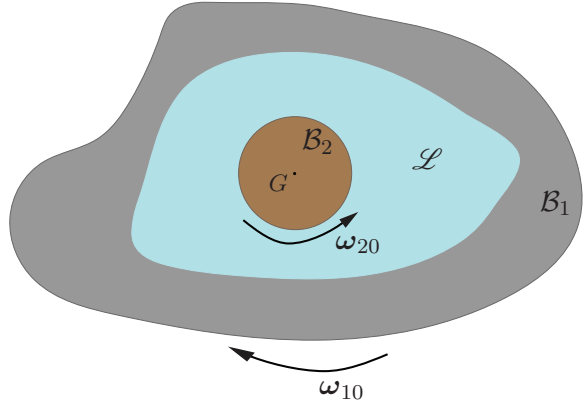}
\caption{Initial configuration for a system of rigid bodies with a liquid-filled gap. In this pictorial situation, the motion of the whole system is driven by the initial velocity imparted on the liquid and the initial angular velocities $\V \omega_{10}$ and $\V \omega_{20}$ of $\mathcal B_1$ and $\mathcal B_2$, respectively. }\label{fig:liquid_gap}
\end{figure} 

This type of fluid-solid interaction problems have been widely studied in connection to some geophysical problems related to the motion of the Earth's inner (solid and liquid) core and its influence on the geodynamo (i.e., the mechanism responsible for the generation of Earth's magnetic field and its maintenance against the Ohmic dissipation), see \cite{Ka,Proud,Bu1,Bu4,ScCDV,Va2004}. From the mathematical point of view, there have been several contributions aimed at proving the existence of solutions to the relevant equations of motion and analyzing their stability properties. In the case where no rigid core is within the liquid-filled cavity, it was conjectured by Zhukovskii (\cite{Zh}) and rigorously proved by the present author and collaborators  that the liquid has a {\em stabilizing effect} on the motion of the solid (see \cite{Ma,Ma2,DiGaMaZu,MaPrSi19,MaPrSi}). In fact, there exists a finite interval of time (whose length depends on the liquid viscosity) where the motion of the system has a “chaotic” nature (as shown numerically in \cite{DiGaMaZu} and experimentally in \cite{Ma2}). After this interval of time, the system reaches (at an exponentially fast rate) a more orderly configuration, corresponding to a steady state in which the system moves as a whole rigid body with a constant angular velocity (see \cite{MaPrSi19,MaPrSi} for a rigorous mathematical proof of this phenomena when the liquid is subject to no-slip and partial slip boundary conditions, respectively). 

Concerning the motion of solids with fluid-filled gaps, known results mainly focus on the translational and rotational motions of rigid bodies in a liquid occupying a bounded domain with a {\em prescribed motion of the liquid outer boundary}. The works \cite{sather,FuSa,Sau} provide the first results of existence of weak solutions \`a la Leray-Hopf to the Navier-Stokes equations in bounded regions with moving boundaries. For the fluid-solid interaction problems with a finite number of rigid bodies within a liquid, existence of weak solutions {\em up to collisions} are proved in \cite{DeEs,CoMaTu}. The work \cite{GrMa} deals with local strong solutions, whereas \cite{Tuc,gunzburger,Fe2003,Fei2003} provide the first results of existence of global weak solutions for both incompressible and compressible cases. We refer also to \cite{Hillairet,Hillairet2009,Hillairet2015,Chemetov2017}
where different boundary conditions and regularity of the boundary are considered for the global existence theory, and to \cite{Sueur2015} where uniqueness of Leray-Hopf solutions in the 2D case is proved.  

In this paper, we show that the problem of free rotations of rigid bodies with a liquid-filled gap admits global weak solutions \`a la Leray-Hopf. In addition, we determine the largest space of initial data for which the equations of motion are well-posed in the setting of maximal $L^p-L^q$ regularity and time-weighted $L^p$ spaces. It is worth emphasizing that for the problem at hand, possible translations of the solids are disregarded. This simplifying assumption has to be contrasted with the existing (cited above) literature in which the motion of the outer solid is instead prescribed. The novelty of the paper lies on considering the full moving boundary problem\footnote{Note that different portions ($\mathcal C$ and $\mathcal S$) of the liquid boundary move with different (unknown) motion. } and proving the existence of weak solutions (Theorem \ref{th:weak}) together with  important properties like a Serrin-type result for weak-strong uniqueness (Theorem \ref{th:continuous_dependence}). One of the main objective of this work is to show that, similarly to the case when no solid is within the liquid-filled cavity, the fluid has a {\em stabilizing effect on the motion of both solids}. In fact, it will be shown that the long-time dynamics\footnote{The long-time behavior of solutions to the governing equations in the Leray-Hopf class for any initial data with finite kinetic energy.} of the whole system is completely characterized by the rest state for the liquid and solid cores relatively to the outer solid, and the system moving as a whole rigid body (see \eqref{eq:decay0}). Such stabilization property is obtained for a large class of fluid-solids configurations. In particular, no restriction will be imposed on the initial data or on physical properties like the Reynolds number, the mass distribution in the outer solid or on the size of the inner core. As an example, take the situation depicted in Figure \ref{fig:liquid_gap} as initial configuration. Note that the initial angular velocities of the solids are in opposite directions (in fact, they could be around any axis). If there was no liquid in the gap between the solids, the motion of the thick crust and inner core would be completely decoupled. When a viscous incompressible fluid fills that gap, the eventual motion of the whole system will be a rigid body motion with crust and inner core at relative rest. From the mathematical point of view, this effect is captured by introducing the new variable $\V \omega$ (see \eqref{eq:new_variables}), the equivalent formulation \eqref{eq:Motion} and by proving the decay \eqref{eq:decay0}. Finally, we prove a local well-posedness result (Theorem \ref{th:strong}) in the functional setting of maximal $L^p-L^q$ regularity in time-weighted $L^p$ spaces. This result is the first of the kind for this class of fluid-solid interactions. 

Here is the plan of the paper. After presenting the basic notation and recalling a well-known Gr\"onwall-type lemma, we proceed with Section \ref{sec:preliminary_formulation} containing the mathematical formulation of the problem as the coupled system of differential equations \eqref{eq:motion}, given by the Navier-Stokes equations and the balances of angular momentums of $\mathcal B_1$ and $\mathcal B_2$, respectively. In Section \ref{sec:functional}, we introduce our functional setting. As equations \eqref{eq:motion} involve both differential and integral terms, in Section \ref{sec:equivalent_formulation}, we provide an equivalent formulation of the problem by replacing the (physical) equations of motion with those governing the motion of a rigid body with a cavity completely filled by a viscous impressible fluid with varying density. In Section \ref{sec:weak}, we prove the existence of weak solutions and related properties. In Section \ref{sec:strong}, we demonstrate the existence of strong solution in the $L^p-L^q$ setting. 

The notation used throughout this paper is quite standard. $\N$ denotes the set of natural numbers. $\R$ indicates the set of real numbers, and $\R^n$ the Euclidean $n$-dimensional space equipped with the canonical basis  $\{\V e_1, \V e_2,\dots, \V e_n\}$. The components of a vector $\V{\mathsf v}$ with respect to the canonical basis are indicated by $(\mathsf v_1,\mathsf v_2,\dots,\mathsf v_n)$, whereas $|\V{\mathsf  v}|$ represents the magnitude of $\V{\mathsf v}$. We will use the Einstein convention for the summation of dummy indexes, and ``$:$'' will denote the tensor contraction. Moreover,  $B_R(G)$ denotes the ball  in $\R^3$ with center at a point $G\in \R^3$ and radius $R$. The ball centered at the origin of a coordinates system $\{\V e_1, \V e_2, \V e_3\}$ will be simply denoted by $B_R$. 
 
If $A$ is an open set of $\R^n$, $s\in \R$ and $p\in [1,\infty]$, then $L^p(A)$, $W^{s,p}(A)$, $W^{s,p}_0(A)$ denote the Lebesgue and (generalized) Sobolev spaces, with norms $\norm{\cdot}_{L^p(A)}$ and $\norm{\cdot}_{W^{s,p}(A)}$, respectively\footnote{Unless confusion arises, we shall use the same symbol for spaces of scalar, vector and tensor functions. }. 

For a bounded, Lipschitz domain $A$, with outward unit normal $\V n$,  we will often use the following well-known Helmholtz-Weyl decomposition (e.g., \cite[Section III.1]{Ga}): 
\begin{equation}\label{eq:HW}
L^q(A) = H_q(A) \oplus G_q(A), 
\end{equation}
where $q\in (1,\infty)$, $H_q(A):=\{\V u\in L^q(A):\; \diver \V u=0\text{ in }A\text{, and } \V u\cdot \V n=0 \text{ on }\partial A\}$ ($\diver \V u$ and $\V u \cdot \V n$ have to be understood in the sense of distributions), and $G_q(A):=\{w \in L^q(A):\; w=\nabla \pi, \text{ for some }\pi\in W^{1,q}(A)\}$. In the case of $q=2$, we will simply write $H(A)$ and $G(A)$, respectively. 

If $(X, \norm{\cdot}_X)$ is a Banach space, for an interval $I$  in $\R$ and $1\le p<\infty$, $L^p(I;X)$ (resp. $W^{k,p}(I;X)$, $k\in \N$) will denote the space of functions $f$ from $I$ to $X$ for which $\left(\int_I \norm{f(t)}^p_X\; \d t\right)^{1/p}<\infty$ (resp. $\sum^k_{\ell=0}\left(\int_I \norm{\partial^\ell_t f (t)}^p_X\; \d t\right)^{1/p}<\infty$). Similarly, $C^k(I;X)$ indicates the space of functions which are $k$-times differentiable with values in $X$, and having $\max_{t\in I}\norm{\partial^\ell_t \cdot}_X < \infty$, for all $\ell = 0,1,...,k$. Finally, $C_w(I;X)$ is the space of functions $f$ from $I$ to $X$ such that  that the map $t \in I \mapsto \phi(f(t))\in \R$ is continuous for all bounded linear functionals $\phi$ defined on $X$. 

We conclude this section by recalling the following Gr\"onwall-type lemma that will be used in the paper. For its proof, we refer the interested reader to \cite{Ma2}. 
\begin{lemma}\label{lem:gronwall1}
Suppose that a function $y\in L^\infty(0,\infty)$, $y\ge 0$, satisfies the following inequality for a.~a. $s\ge 0$ and all $t\ge s$: 
\begin{equation*} 
y(t)\le y(s)-k\int_s^ty(\tau)\,d\tau+\int_s^tF(\tau)\,d\tau\,.
\end{equation*}
Here, $k>0$, and $F\in L^q(a,\infty)\cap L^1_{{\rm loc}}(0,\infty)$, for some $a>0$ and $q\in [1,\infty)$, satisfies $F(t)\ge 0$ for a.~a. $t\ge 0$. Then
$$
\lim_{t\to\infty}y(t)=0\,.
$$ 
If $F\equiv 0$, then 
$$
y(t)\le y(s)\,{\rm e}^{-k(t-s)}
\,,\ \ \mbox{for all $t\ge s$}\,.
$$
\end{lemma}

We are now ready to introduce the equations governing the motion of the system of rigid bodies with a liquid-filled gap. 
 
\section{A preliminary mathematical formulation of the problem}\label{sec:preliminary_formulation}

Consider $\mathcal B_1:=\mathcal V_1\setminus \overline{\mathcal V}$, with $\mathcal V_1$ and $\mathcal V$  bounded domains in $\R^3$, $\overline{\mathcal V}\subset \mathcal V_1$, $\overline{B_R(G)}\subset \mathcal V$, and $\mathcal B_2:=B_R(G)$. Let us denote $\mathcal C:=\partial \mathcal V$, $\mathcal S:=\partial \mathcal B_2$~\footnote{$\mathcal S$ is the sphere in $\R^3$ centered at $G$ with radius $R$.}, and ${\Ell}:=\mathcal V\setminus \overline{B_R(G)}$ be the volume occupied by the liquid at each time. Throughout the paper, we will assume that ${\Ell}$ is of class $C^2$. 

Let $\mathcal{F}\equiv\{G, \V e_1,\V e_2,\V e_3\}$ be the {\em non-inertial} reference frame with origin at $G$, and axes coinciding with the {\em central  axes of inertia} of the coupled system $\mathscr S_C$; these axes are directed along the eigenvectors of the inertia tensor $\T I_{C}$ of $\mathscr S_C$ with respect to $G$, and with corresponding (positive and time-independent) eigenvalues $\lambda_1$, $\lambda_2$, and $\lambda_3$ (also called {\em central moment of inertia}). Let us denote by $\T I_{\mathcal B}$ the inertia tensor of the rigid body $\mathcal B_1$ with respect to $G$. Since $\mathcal B_2$ is a homogeneous rigid ball with center at $G$, then any axis passing through its center is also a central axis of inertia. Thus, the inertial tensor of $\mathcal B_2$ with respect to $G$ is simply  
$\lambda (\V e_1\otimes \V e_1+\V e_2\otimes \V e_2+\V e_3\otimes \V e_3)$ with $\lambda=2/5\; mR^2$, and $m$ the mass of the rigid ball. With respect to the reference frame $\mathcal F$, all the volumes considered above are  time-independent. 

The following system of differential equations describes  the dynamics of the given system 
in the reference frame $\mathcal F$ 
\footnote{We refer to 
\cite{Ma}, and \cite{Ma2} for more details about this kind of formulation obtained  
for similar problems in liquid-solid interactions. }. 
\begin{equation}\label{eq:motion}
\begin{aligned}
&\left.\begin{split}
&\rho\left(\frac{\partial \V u}{\partial t}+\V v\cdot \nabla \V u+\V \omega_1\times \V u\right)
=\diver \T T(\V u,p)
\\
&\diver \V u=0
\end{split}\right\}&&\text{ on }\Ell\times (0,\infty),
\\
&\T I_{\mathcal B}\cdot\dot{\V \omega}_1 +\V \omega_1\times \T I_{\mathcal B}\cdot \V \omega_1
=-\int_{ \mathcal C}\V x\times \T T(\V u,p)\cdot \V n\; \d \sigma \qquad &&\text{ in }(0,\infty),
\\
&\lambda(\dot{\V \omega}_2 +\V \omega_1\times \V \omega_2)
=-\int_{\mathcal S}\V x\times \T T(\V u,p)\cdot \V n\; \d \sigma \qquad &&\text{ in }(0,\infty),
\\
&\V u=\V \omega_1\times \V x\qquad &&\text{ on }\mathcal C,
\\
&\V u=\V \omega_2\times \V x\qquad &&\text{ on }\mathcal S.
\end{aligned}
\end{equation}
Here, $\V u,p, \mu$ and $\rho$ denote the Eulerian absolute velocity and pressure of the liquid, its shear viscosity and (constant) density, respectively. In addition,  $\V v$ indicates the Eulerian velocity of the liquid relative to $\mathcal B_1$
\begin{equation}\label{eq:relative_velocity}
\V v:=\V u-\V \omega_1\times \V x. 
\end{equation}
We notice that $\diver \V v=0$, and it enjoys the following boundary conditions 
\begin{equation}\label{eq:bc_v}
\V v=\V 0\qquad \text{on }\mathcal C,\quad \text{and }\quad \V v\cdot \V n=0\qquad \text{on }\mathcal S. 
\end{equation}
Moreover, $\T T(\V u,p)$ denotes the Cauchy stress 
tensor for a viscous incompressible fluid
\begin{equation}\label{eq:Cauchy_stress}
\T T(\V u,p):=-p\T 1+2\mu \T D(\V u), \qquad
\text{where } \; \T D(\V u):=\frac 12 (\nabla \V u+(\nabla \V u)^T).  
\end{equation}
Finally, $\V \omega_1$  and $\V \omega_2$ are the angular velocities of $\mathcal B_1$ and $\mathcal B_2$, respectively. Equations \eqref{eq:motion}$_{1,2}$ with \eqref{eq:relative_velocity} and \eqref{eq:Cauchy_stress} are the {\em Navier-Stokes equations} in the non-inertial reference frame $\mathcal F$. These equations describe the dynamics of the liquid. Equations \eqref{eq:motion}$_{3,4}$ are the {\em balances of  angular momentum} (with respect to $G$) of $\mathcal B_1$ and $\mathcal B_2$, respectively. In particular, the surface integrals in \eqref{eq:motion}$_{3,4}$ represent the total torque exerted by the liquid on the cavity surface $\mathcal C$ and on the sphere $\mathcal S$, respectively. The equations of motion are augmented with the {\em no-slip} boundary conditions \eqref{eq:motion}$_{5,6}$ at $\mathcal C$ and $\mathcal S$, respectively.   

Equations \eqref{eq:motion} feature a combination of {\em dissipative} and {\em conservative} components. The {\em dissipative} role is played by the liquid variable through equations \eqref{eq:motion}$_{1,2,5,6}$.  Whereas, the {\em conservative} feature comes from the coupling with the equations \eqref{eq:motion}$_{3,4}$ describing the dynamics of the solids. As a matter of fact, the energy dissipates only in the liquid variable (see equation \eqref{eq:energy} below), and the total angular momentum (with respect to $G$) of the whole system is conserved at all times (see equation \eqref{eq:conservation0} below). These properties are satisfied for ``sufficiently regular'' solutions. 

\begin{lemma}[Energy Balance] \label{lem:energy}
Consider $t_0\ge 0$, and assume that the quadruple $(\V u, p, \V \omega_1,\V \omega_2)$  
satisfies the following regularity properties for all $T>0$:
\begin{equation}\label{eq:regularity}\begin{split}
&\V u\in C^0([t_0, t_0+T];W^{1,2}(\Ell)\cap H(\Ell))
\cap L^2(t_0,t_0+T;W^{2,2}(\Ell)),
\\
&\quad\frac{\partial \V u}{\partial t}\in L^2(t_0,t_0+T;L^2(\Ell)),\;
\quad p\in L^2(t_0,t_0+T;W^{1,2}(\Ell)),\; 
\\
&\qquad \qquad\qquad\qquad\V \omega_1,\V \omega_2 \in  W^{1,\infty}(t_0,t_0+T). 
\end{split}\end{equation}  
If $(\V u,p,\V \omega_1,\V \omega_2)$ satisfies \eqref{eq:motion} a.e. in $(t_0,\infty)$, then 
the following {\em energy balance} holds. 
\begin{equation}\label{eq:energy}
\frac 12 \frac{\d }{\d t}\left[\rho \norm{\V u}_{L^2(\Ell)}^2+\V \omega_1\cdot \T I_{\mathcal B}\cdot \V \omega_1
+\lambda |\V \omega_2|^2\right]+2\mu\norm{\T D(\V u)}^2_{L^2(\Ell)}=0. 
\end{equation}
\end{lemma}

\begin{proof}
Let  us take the $L^2$-inner product of \eqref{eq:motion}$_2$ with $\V u$, we find that 
\[
\frac \rho2 \frac{\d }{\d t}\norm{\V u}_{L^2(\Ell)}^2+\int_{\Ell}(\V v\cdot \nabla \V u)\cdot \V u\; \d V
-\int_{\Ell}\V u\cdot \diver\T T\; \d V=0.
\]
Since $\diver \V v=\diver \V u=0$ by \eqref{eq:motion}$_2$, 
using \eqref{eq:bc_v} and Gauss' Theorem, we can infer the following 
\[
\int_{\Ell}(\V v\cdot \nabla \V u)\cdot \V u\; \d V=0.
\]
By \eqref{eq:motion}$_{5,6}$ and \eqref{eq:Cauchy_stress}, and again by Gauss' Theorem, we get 
\[
\frac \rho2 \frac{\d }{\d t}\norm{\V u}_{L^2(\Ell)}^2
-\V \omega_1\cdot\int_{\mathcal C}\V x\times \T T\cdot \V n\; \d \sigma 
-\V \omega_2\cdot\int_{\mathcal S}\V x\times \T T\cdot \V n\; \d \sigma
+2\mu\norm{\T D(\V u)}^2_{L^2(\Ell)}=0.
\]   
From the latter displayed equation, \eqref{eq:energy} immediately follows by using 
\eqref{eq:motion}$_{3,4}$ dot-multiplied by $\V \omega_1$ and $\V \omega_2$, respectively.  
\end{proof}

With the same hypotheses of the previous lemma, we can show the following. 
\begin{lemma}[Conservation of total angular momentum]\label{lem:conservation}
If the quadruple $(\V u,p,\V \omega_1,\V \omega_2)$ satisfies \eqref{eq:regularity} for some $t_0\ge 0$, 
and \eqref{eq:motion} a.e. in $(t_0,\infty)$, then 
\begin{equation}\label{eq:balance_angular_momentum}
\V{\dot A}+\V \omega_1\times \V A=\V 0,
\end{equation}
where 
\begin{equation}\label{eq:anugular_momentum}
\V A:=\rho \int_{\Ell} \V x\times \V u\; \d V+\T I_{\mathcal B}\cdot \V \omega_1+\lambda \V \omega_2
\end{equation}
is the {\em total angular momentum} of the whole system with respect to $G$. In particular, equation \eqref{eq:balance_angular_momentum} implies that 
\begin{equation}\label{eq:conservation0}
|\V A(t)|=|\V A(t_0)|,\qquad \text{all }t\ge t_0.
\end{equation}
\end{lemma}

\begin{proof}
From \eqref{eq:motion}$_{1,3,4}$, we find that 
\begin{equation}\label{eq:dotA}
\begin{split}
\V{\dot A}&=\rho \int_{\Ell} \V x\times \frac{\partial \V u}{\partial t}\; \d V
+\T I_{\mathcal B}\cdot \dot{\V \omega}_1+\lambda \dot{\V \omega}_2
\\
&=\int_{\Ell} \V x\times\left(\diver \T T(\V u,p)-\rho\V v\cdot \nabla \V u
-\rho\V \omega_1\times \V u\right)\; \d V -\V \omega_1\times \T I_{\mathcal B}\cdot \V \omega_1
\\
&\qquad -\int_{ \mathcal C}\V x\times \T T(\V u,p)\cdot \V n\; \d \sigma
-\lambda \V \omega_1\times \V \omega_2
-\int_{\mathcal S}\V x\times \T T(\V u,p)\cdot \V n\; \d \sigma. 
\end{split}
\end{equation}
Since the Cauchy stress tensor is symmetric, by Gauss' Theorem we get that 
\[
\int_{\Ell} \V x\times\diver \T T(\V u,p)\; \d V
-\int_{ \mathcal C}\V x\times \T T(\V u,p)\cdot \V n\; \d \sigma
-\int_{\mathcal S}\V x\times \T T(\V u,p)\cdot \V n\; \d \sigma=\V 0.
\]
Using again Gauss' Theorem together with \eqref{eq:bc_v}, we also find that 
\[\begin{split}
-\rho\int_{\Ell} \V x\times\left(\V v\cdot \nabla \V u+\V \omega_1\times \V u\right)\d V
&=\rho\int_{\Ell}\left[\V u\times (\V \omega_1\times \V x)
+ \V x\times\left(\V u\times \V \omega_1\right)\right]\d V
\\
&=-\rho\int_{\Ell}\V \omega_1\times (\V x\times \V u)\d V.
\end{split}\]
In the last equality, we have used the following property of the cross product in $\R^3$: 
\[
\V a\times (\V b\times \V c)+\V b\times (\V c\times \V a)=-\V c\times (\V a\times \V b),\quad 
\text{all } \V a,\V b,\V c\in \R^3.
\]
Therefore, \eqref{eq:dotA} becomes 
\[
\V{\dot A}=-\V \omega_1\times \left(\rho\int_{\Ell}\V x\times \V u\;\d V
+\T I_{\mathcal B}\cdot \V \omega_1+\lambda \V \omega_2\right)=-\V \omega_1\times \V A. 
\]
This shows \eqref{eq:balance_angular_momentum}, from which \eqref{eq:conservation0} 
immediately follows by taking the dot-product of \eqref{eq:balance_angular_momentum} 
by $\V A$. 
\end{proof}

In the next section, we will provide the functional setting in which we will study the existence of solutions to the equations of motion. 

\section{Functional spaces}\label{sec:functional}

Consider the spaces 
\[\begin{aligned}
\mathcal R(\mathcal V)&:=\{\V u\in C^\infty(\mathcal V):\; \V u=\V \omega_u \times \V x\text{ on }\mathcal V,
\text{ for some }\V \omega_u \in \R^3\},
\\
C^\infty_R(\mathcal V)&:=\left\{\V u\in C^\infty(\mathcal V):\; \V u=\V \omega_u \times \V x\text{ in a neighborhood of }\mathcal B_2,
\text{ for some }\V \omega_u \in \R^3\right\}. 
\end{aligned}\]
For every $1\le q<\infty$, let us consider the norm 
\begin{equation}\label{eq:norm_w}
\norm{\V u}_q:=\left(\int_{\mathcal V}\tilde \rho\V u^q\right)^{1/q}
=\left(\rho\norm{\V u}^q_{L^q(\Ell)}+\lambda|\V \omega_u|^q\right)^{1/q},\qquad\text{for all }\V u\in C^\infty_R(\mathcal V). 
\end{equation}
In the above equation,  
\begin{equation}\label{eq:density}
\tilde \rho:=\left\{\begin{split}
\rho\qquad\quad &\text{on }\Ell
\\
\frac{15\lambda}{8\pi R^5}\qquad &\text{on }\mathcal B_2.
\end{split}\right.
\end{equation}
$L^q_R(\mathcal V)$ indicates the completion of $C^\infty_R(\mathcal V)$ in the norm $\norm{\cdot}_q$. In the particular case of $q=2$, $L^2_R(\mathcal V)$ is a Hilbert space endowed with the inner product 
\begin{equation}\label{eq:inner_product_w}
(\V u,\V v):=\int_{\mathcal V}\tilde \rho\V u\cdot \V v=\int_\Ell \rho\; \V u\cdot \V v+\lambda \V \omega_u\cdot \V \omega_v.
\end{equation}
One can show that the following characterization holds for every $1\le q<\infty$ (see e.g. \cite[Chapter 1, Section 1]{temam1985})
\[
L^q_R(\mathcal V)=\{\V u\in L^q(\mathcal V):\; \V u=\V \omega_u\times \V x\;\text{ on }\mathcal B_2\;\text{ for some }\V \omega_u\in \R^3\}.
\]

Consider the spaces 
\[
\mathcal D_R(\mathcal V):= \{\V u\in C^\infty_R(\mathcal V)\cap C^\infty_0(\mathcal V):\; \diver \V u=0\; \text{ on }\mathcal V\}, 
\] 
and for $T>0$
\begin{multline*}
\mathcal D_R(\mathcal V_T):= \{C^\infty_0(\mathcal V\times[0,T)):\; \diver \V u=0\; \text{ on }\mathcal V\times[0,T),
\\
\qquad\qquad\quad\V u=\V \omega_u \times \V x\text{ in a neighborhood of }\mathcal B_2, \text{ for some }\V \omega_u \in C^\infty_0([0,T))\}.
\end{multline*}

In addition, $\mathcal H_q(\mathcal V)$ denotes the completion of $\mathcal D_R(\mathcal V)$ with respect to the norm $\norm{\cdot }_q$. In a similar fashion to the classical space of the hydrodynamics (see e.g. \cite[Section III.2]{Ga}), one can show that, the space $\mathcal H_q(\mathcal V)$ has the following representation 
\[\begin{split}
\mathcal H_q(\mathcal V)=\{\V u\in L^q_R(\mathcal V):\;& \diver \V u=0\;\text{ on }\mathcal V,\; \V u\cdot \V n=0\;\text{ on }\mathcal C\}.
\end{split}\]
Moreover, we can consider the projection operator $\mathcal P_q$ of $L^q_R(\mathcal V)$ onto $\mathcal H^q(\mathcal V)$ (c.f. \cite[Remark III.1.1 \& Theorem III.1.2]{Ga}). Let $1<q<\infty$. The space $\mathcal H^1_q(\mathcal V)$ denotes the completion of $\mathcal D_R(\mathcal V)$ with respect to the norm 
\begin{equation}\label{eq:norm_1q}
\norm{\cdot}_{1,q}:=\left(\norm{\cdot}_q^q+2\mu\norm{\T D(\cdot)}^q_{L^q(\mathcal V)}\right)^{1/q}.
\end{equation}
The right-hand side of latter displayed equation defines indeed a norm due to the following {\em Korn inequality} (\cite[Theorem 1]{Geymonat}). 
\begin{lemma}\label{lem:korn_q}
For $1 <q<\infty$ the space $U_q:=\{\V u\in L^q(\mathcal V):\; \T D(\V u)\in L^q(\mathcal V)\}$ is equal to $W^{1,q}(\mathcal V)$. Moreover, there exist two constants $0<c_1<c_2$ such that
\[
c_1\norm{\V u}_{W^{1,q}(\mathcal V)}\le\left(\norm{\V u}_{L^q(\mathcal V)}^q+\norm{\T D(\V u)}^q_{L^q(\mathcal V)}+\norm{\diver(\V u)}^q_{L^q(\mathcal V)}\right)^{1/q} \le c_2\norm{\V u}_{W^{1,q}(\mathcal V)},
\]
for all $\V u\in U_q$. 
\end{lemma}
The following characterization holds 
\begin{equation*}
\mathcal H^1_q(\mathcal V)=\{\V u\in W^{1,q}_0(\mathcal V):\; \diver \V u=0\;\text{ on }\mathcal V,\; 
\V u=\V \omega_u\times \V x\;\text{ on }\mathcal B_2\;\text{ for some }\V \omega_u\in \R^3\}.
\end{equation*}
We notice that $\mathcal D_R(\mathcal V)\subset \mathcal H^1_q(\mathcal V)$, so $\mathcal H^1_q(\mathcal V)$ is dense in $\mathcal H_q(\mathcal V)$. Moreover, since $W^{1,q}(\mathcal V)$ is compactly embedded in $L^q(\mathcal V)$ for all $\displaystyle 1\le q<\infty$ (\cite[Theorem 6.3]{Adams}), we have the following lemma. 
\begin{lemma}\label{lem:embedding1}
If $\displaystyle 1\le q<\infty$, then the embedding of $\mathcal H^1_q(\mathcal V)$ in $\mathcal H_q(\mathcal V)$ is compact. 
\end{lemma}

We are now in position to state some inequalities that will be used in the next sections. The proof are standards and will be omitted. We start with the following {\em Korn-type equalities}. 

\begin{lemma}[Korn's equality in $\mathcal H^1_2$]
For all $\V v, \V w\in \mathcal H^1_2(\mathcal V)$ the following equality holds 
\[
2\int_{\Ell}\T D(\V v):\T D(\V w)\; \d V
=\int_{\mathcal V}\nabla \V v:\nabla \V w\; \d V. 
\]
In particular, \begin{equation}\label{eq:korn_2}
\norm{\nabla \V v}_{L^2(\mathcal V)}=\sqrt 2 \norm{\T D( \V v)}_{L^2(\Ell)}. 
\end{equation}
\end{lemma}

In a similar fashion as in \cite[Proposition 3.]{Geymonat}, and applying Lemma \ref{lem:korn_q} to \eqref{eq:norm_1q} one can easily show the following Poincar\'e-Korn inequality. 
\begin{lemma}[Poincar\'e-Korn inequality in $\mathcal H^1_q$]
Let $1<q<\infty$. There exist two positive constants $k_1<k_2$ such that 
\begin{equation}\label{eq:korn_q}
k_1\norm{\V v}_{W^{1,q}(\mathcal V)}\le \norm{\V v}_{1,q}\le k_2\norm{\T D(\V v)}_{L^q(\Ell)},\qquad \text{for all }\;\V v\in \mathcal H^1_q(\mathcal V). 
\end{equation}
\end{lemma}

Recall that $\mathcal V=\Ell\cup \overline{\mathcal B_2}$. Next lemma follows directly from \eqref{eq:korn_2} and Poincar\'e inequality. 
\begin{lemma}\label{le:estimates_2}
The following estimates hold for all $\V v\in \mathcal H^1_2(\mathcal V)$. \begin{enumerate}
\item Let $\V \omega_v \in \R^3$ be such that $\V v=\V \omega_v \times \V x$ on $\mathcal B_2$, then 
\begin{equation}\label{eq:korn1}
\norm{\nabla \V{ v}}_{L^2(\Ell)}^2+\frac 83 \pi R^3 |\V \omega_v|^2=2\norm{\T D (\V{v})}_{L^2(\Ell)}^2
\end{equation}
\item There exists a positive constants $C_1$ depending only on $\Ell$ (and independent of $\V v$) such that 
\begin{equation}\label{eq:poincare_korn}
\norm{\V v}_{L^2(\Ell)}\le C_1\norm{\T D(\V v)}_{L^2(\Ell)}.
\end{equation}
\end{enumerate}
\end{lemma}

Since $\mathcal D_R(\mathcal V)$ is dense in $\mathcal H^1_R(\mathcal V)$, by  Sobolev inequality together with \eqref{eq:korn_q}, we can prove the following lemma.  

\begin{lemma}\label{le:sobolev_korn}
For all $s<3$, there exists a positive constant $k$ depending only on $\Ell$ (and independent of $\V v$) such that 
\begin{equation}\label{eq:sobolev_korn}
\norm{\V v}_{q}\le k\norm{\T D(\V v)}_{L^s(\Ell)}, \qquad \text{for all }\; \V v\in \mathcal H^1_q(\mathcal V)
\end{equation}
if and only if $q=6/(3-s)$. 
\end{lemma}

We conclude this section by introducing the space $\mathcal H^k_q(\mathcal V)$ as the completion of $\mathcal D_R(\mathcal V)$ with respect to the norm $\norm{\cdot }_{W^{k,q}(\mathcal V)}$ for all $1\le q<\infty$ and $k\in \N$, $k\ge 2$. In particular, $\mathcal H^2_q(\mathcal V)$ is a Banach space endowed with the norm 
\begin{equation}\label{eq:norm_2q}
\norm{\cdot}_{2,q}:=\left(\norm{\cdot}^q_{L^q(\mathcal V)}+\norm{\T D(\cdot)}^q_{L^q(\mathcal V)}+\norm{\T H(\cdot)}^q_{L^q(\mathcal V)}\right)^{1/q},
\end{equation}
where $\T H$ denote the third order tensor of second order derivatives. Similarly to Lemma \ref{lem:embedding1}, the following embedding also holds. 
\begin{lemma}\label{lem:embeddingk}
If $\displaystyle 1\le q<\infty$ and $k\ge 1$, then the embedding of $\mathcal H^k_q(\mathcal V)$ in $\mathcal H_q(\mathcal V)$ is compact. 
\end{lemma}

The previous results together with Lemma \ref{lem:energy} and Lemma \ref{lem:conservation} 
allow us to present a new mathematical formulation of the problem. This new formulation is equivalent to \eqref{eq:motion}, and will reveal more features of the dynamics of our physical system. 
 
\section{An equivalent formulation}\label{sec:equivalent_formulation}

Let us introduce the new variable
\begin{equation}\label{eq:new_variables}
\V \omega:=\V \omega_2-\V \omega_1.  
\end{equation}
The definition of the variable $\V \omega$ comes from the following heuristic reasoning. Due to the liquid viscosity (since also $\T D(\V u)=\T D(\V v)$), we expect the velocity of the liquid relative to $\mathcal B_1$ (and also the one relative to $\mathcal B_2$) to decay to zero as time approaches to infinity. If this happens, from the boundary conditions \eqref{eq:bc_v}, also $\V \omega$ is expected to decay, and  the system would then move as a whole rigid body.  

Let $\T I:=\T I_{C}+\lambda \T 1=(\lambda_1+\lambda)\V e_1\otimes\V e_1
+(\lambda_2+\lambda)\V e_2\otimes\V e_2+(\lambda_3+\lambda)\V e_3\otimes\V e_3$ be the inertia tensor of the whole system with respect to $G$. Here, $\T 1$ denotes the identity tensor in $\R^3\times \R^3$. We note that  $\T I_{C}=\T I_{\Ell}+\T I_{\mathcal B}$, where
\[
\V b\cdot \T I_{\Ell}\cdot \V c=\rho\int_{\Ell}(\V x\times \V b)\cdot (\V x\times \V c)\; \d V,
\qquad \V b, \V c\in \R^3.
\]
The tensor $\T I$ is a symmetric and positive definite (thus, invertible). To simplify the notation, let us introduce the vector field  
\begin{equation}\label{eq:omega_R}
\V \omega_R:=-\T I^{-1}\cdot\left[\rho\int_{\Ell}\V x\times \V v\; \d V+\lambda \V \omega\right].
\end{equation}

In terms of the variables $(\V v,p,\V \omega_1,\V \omega)$, and taking into account \eqref{eq:Cauchy_stress} together with Lemma \ref{lem:conservation}, the equations of motion \eqref{eq:motion} can be equivalently reformulated as follows: 
\begin{equation}\label{eq:Motion}
\begin{aligned}
&\left.\begin{split}
&\rho\left(\frac{\partial \V v}{\partial t}+\V{\dot \omega}_1\times \V x+
\V v\cdot \nabla \V v+2\V \omega_1\times \V v\right)
\\
&\qquad\qquad\qquad\qquad\quad\qquad
=\frac \rho2 \nabla |\V \omega_1\times \V x|^2+\diver \T T(\V v,p)
\\
&\diver \V v=0
\end{split}\right\}&&\text{on }\Ell\times (0,\infty),
\\
&\T I\cdot(\V{\dot \omega}_1-\V{\dot \omega}_R) +\V \omega_1\times \T I\cdot (\V \omega_1-\V \omega_R)=\V 0 
\qquad&&\text{in }(0,\infty),
\\
&\lambda\left(\V{\dot \omega}+\V{\dot \omega}_1
+\V \omega_1\times \V \omega\right)
=-\int_{\mathcal S}\V x\times \T T(\V v,p)\cdot \V n\; \d \sigma \qquad &&\text{in }(0,\infty),
\\
&\V v=\V 0\qquad &&\text{on }\mathcal C,
\\
&\V v=\V \omega\times \V x\qquad &&\text{on }\mathcal S.
\end{aligned}
\end{equation}
The proof of the equivalence between the formulations \eqref{eq:motion} and \eqref{eq:Motion} goes along the one provided in the case when no rigid body is within the cavity of $\mathcal B_1$ (namely, if $R\equiv 0$). We refer the interested reader to \cite[Appendix]{DiGaMaZu} and \cite[Sections 2.1 and 2.2]{Ma2}. The energy balance \eqref{eq:energy} can be rewritten as follows 
\begin{equation}\label{eq:Energy0}
\begin{split}
\frac 12 \frac{\d }{\d t}\left[\rho \norm{\V v}^2_{L^2(\Ell)}+\lambda |\V \omega|^2-\V \omega_R\cdot \T I \cdot \V \omega_R+(\V \omega_1-\V \omega_R)\cdot \T I\cdot 
(\V \omega_1-\V \omega_R)\right]
+2\mu\norm{\T D( \V v)}^2_{L^2(\Ell)}=0.  
\end{split}\end{equation}

Consider the functionals  
\begin{equation}\label{eq:a_varphi}
\begin{split}
b:\;\V w\in\mathcal H_q(\mathcal V)\mapsto 
b(\V w)&:=-\T I^{-1}\cdot \int_{\mathcal V}\tilde \rho\V x\times \V w
\\
&=-\T I^{-1}\cdot \left(\rho\int_{\Ell}\V x\times \V w+\lambda \V \omega_w\right)\in \R^3,
\end{split}\end{equation}
and taking $q=2$ in the previous definition, we define   
\begin{equation}\label{eq:lyapunov}
\mathcal E:\; \V w\in \mathcal H_2(\mathcal V)\mapsto \mathcal E(\V w):=\norm{\V w}^2_2-b(\V w)\cdot\T I\cdot b(\V w)\in \R.  
\end{equation}
In particular, if we consider the field 
\begin{equation}\label{eq:extension}
\V{\tilde v}:=\left\{\begin{split}
\V v\quad\quad \text{in }&\Ell,
\\
\V \omega \times \V x\quad \text{in }&\mathcal B_2,
\end{split}\right.\end{equation}
and use \eqref{eq:norm_w} and \eqref{eq:omega_R}, we find that  $b(\V{\tilde v})=\V \omega_R$ and 
\begin{equation}\label{eq:energy_d}
\mathcal E(\V{\tilde v})=\rho \norm{\V v}^2_{L^2(\Ell)}+\lambda |\V \omega|^2-\V \omega_R\cdot \T I \cdot \V \omega_R. 
\end{equation}
The following lemma ensures that $\mathcal E$ is a positive definite functional. Actually, it says a little more. 
\begin{lemma}\label{le:kokr}
There exists a constant $c\in (0,1)$ such that  
\begin{equation}\label{eq:kokr}
c\norm{\V w}^2_2\le \mathcal E(\V w)
\le\norm{\V w}^2_2,
\end{equation}
for all $\V w\in \mathcal H_2(\mathcal V)$. Moreover, for every $\V w\in \mathcal H^1_2(\mathcal V)$, there exists a positive constant $C$ such that 
\begin{equation}\label{eq:dissipation}
\mathcal E(\V w)\le C\norm{\T D( \V w)}^2_{L^2(\Ell)}.
\end{equation}
\end{lemma}

\begin{proof}
To prove \eqref{eq:kokr}, we will borrow some ideas from \cite[Section 7.2.3]{KoKr}. Consider the linear operator with finite dimensional range 
\begin{equation}\label{eq:compact}
\mathbb B:\; \V w\in \mathcal H_q(\mathcal V)\mapsto 
(\mathbb B\V w)(\V x):=-b(\V w)\times \V x\in \mathcal R(\mathcal V),
\end{equation}
where $b(\cdot)$ has been defined in \eqref{eq:a_varphi}. 

If $q=2$,  $\mathbb B$ is a nonnegative self-adjoint operator in $\mathcal H_2(\mathcal V)$ endowed with the inner product $(\cdot,\cdot)$ defined in \eqref{eq:inner_product_w}. In fact, since $\T I$ is symmetric, for all $\V w$ and $\V z\in \mathcal H_2(\mathcal V)$ we have 
\[\begin{split}
(\mathbb B\V w,\V z)&=-\rho\int_{\Ell}(b(\V w)\times \V x)\cdot\V z\; \d V-\lambda \V \omega_z\cdot b(\V w)
\\
&=-b(\V w)\cdot\left(\rho\int_{\Ell}\V x\times \V z\; \d V+\lambda \V \omega_z\right)
=b(\V w)\cdot \T I\cdot b(\V z)=(\V w,\mathbb B\V z).
\end{split}\]
In particular, since $\T I$ is positive definite,  
$
(\mathbb B\V w, \V w)=b(\V w)\cdot \T I\cdot b(\V w)\ge 0$. Moreover, 
\begin{equation}\label{eq:1-B}
((\T 1-\mathbb B)\V w,\V w)=\norm{\V w}^2_2-b(\V w)\cdot \T I\cdot b(\V w)
=\mathcal E(\V w).
\end{equation}
The inequality on the right-hand side of \eqref{eq:kokr} follows immediately from the latter displayed equations. Thus, to complete the proof of \eqref{eq:kokr}, it is enough to show that the operator $\T 1-\mathbb B$ admits a bounded inverse in $( \mathcal H_2(\mathcal V),\norm{\cdot}_2)$. First, we will show that $\T 1-\mathbb B$ is a nonnegative operator on $ \mathcal H_2(\mathcal V)$.

Using the above calculations, we have the following:  
\[\begin{split}
((\T 1-\mathbb B)\V w,\V w)&=\norm{\V w}^2_2-b(\V w)\cdot \T I\cdot b(\V w)
\\
&=\norm{\V w+b(\V w)\times \V x}^2_2-\rho\norm{b(\V w)\times \V x}^2_{L^2(\Ell)}
\\
&\quad-2b(\V w)\cdot \left(\rho\int_{\Ell}\V x\times \V w\; \d V+\lambda \V \omega_w\right) -\lambda |b(\V w)|^2
-b(\V w)\cdot \T I\cdot b(\V w)
\\
&=\norm{\V w+b(\V w)\times \V x}^2_2-b(\V w)\cdot(\T I_{\Ell}+\lambda\T 1)\cdot b(\V w)
+2b(\V w)\cdot \T I\cdot b(\V w)
\\
&\quad -b(\V w)\cdot \T I\cdot b(\V w)
\\
&=\norm{\V w+b(\V w)\times \V x}^2_2-b(\V w)\cdot(\T I_{\Ell}+\lambda\T 1)\cdot b(\V w)
+b(\V w)\cdot \T I\cdot b(\V w)
\\
&=\norm{\V w+b(\V w)\times \V x}^2_2
+b(\V w)\cdot \T I_{\mathcal B}\cdot b(\V w)\ge 0
\end{split}\]
since $\T I_{\mathcal B}=\T I-\T I_{\Ell}-\lambda \T 1$ is also a positive definite tensor.  In addition to this, one can also show that $((\T 1-\mathbb B)\V w,\V w)=0$ iff $\V w\equiv \V 0$ on $\mathcal V$. We need to show only that $((\T 1-\mathbb B)\V w,\V w)=0$ implies that  $\V w\equiv \V 0$ (the converse implication is obvious). If $((\T 1-\mathbb B)\V w,\V w)=0$, then $b(\V w)\cdot \T I_{\mathcal B}\cdot b(\V w)=0$. Since $\T I_{\mathcal B}$ is positive definite, then the previous statement implies that $b(\V w)\equiv \V 0$, and also $(\mathbb B\V w,\V w)=0$. Thus, 
\[
\norm{\V w}_2=(\mathbb B\V w,\V w)=0,
\] 
implying that $\V w\equiv \V 0$ in $\mathcal V$. Summarizing, we have shown that $\mathbb B$ is a linear, nonnegative and self-adjoint operator with finite dimensional range, and for which $\gamma=1$ is not an eigenvalue. Necessarily, $\gamma=1$ is in the resolvent of $\mathbb B$, implying that $\T 1-\mathbb B$ admits a bounded inverse in $\mathcal H_2(\mathcal V)$ endowed with the norm defined in \eqref{eq:norm_w}. This concludes the proof of \eqref{eq:kokr}. 

The estimate \eqref{eq:dissipation} is an immediate consequence of \eqref{eq:kokr}together with \eqref{eq:sobolev_korn}. 
\end{proof}

Using \eqref{eq:energy_d} and \eqref{eq:lyapunov} in \eqref{eq:Energy0}, the balance of energy then reads as follows 
\begin{equation}\label{eq:Energy}
\frac{\d }{\d t}\left[\mathcal E(\V{\tilde v})+(\V \omega_1-\V \omega_R)\cdot \T I\cdot 
(\V \omega_1-\V \omega_R)\right]+4\mu\norm{\T D( \V v)}^2_{L^2(\Ell)}=0,
\end{equation} 
where $\V{\tilde v}$ has been define in \eqref{eq:extension}. From the physical viewpoint, $\mathcal E(\V{\tilde v})+(\V \omega_1-\V \omega_R)\cdot \T I\cdot (\V \omega_1-\V \omega_R)$ represents the {\em total kinetic energy} of the whole system of rigid bodies with a liquid-filled gap. 

Thanks to Lemma \ref{le:kokr}, we can introduce the inner product 
\begin{equation}\label{eq:inner_product_B}
(\V v,\V w)_B:=((\T 1-\mathbb B)\V v,\V w),\qquad \text{for all }\V v,\V w\in \mathcal H_2(\mathcal V)
\end{equation}
with associated norm $\norm{\cdot}_B:=\sqrt{((\T 1-\mathbb B)\cdot,\cdot)}=\sqrt{\mathcal E(\cdot)}$. 

In addition to the energy balance, the conservation of the total angular momentum \eqref{eq:conservation0} for the whole system can be rewritten in terms of the new variables
\begin{equation}\label{eq:conservation}
|\T I\cdot (\V \omega_1(t)-\V \omega_R(t))|=|\T I\cdot (\V \omega_1(0)-\V \omega_R(0))|\qquad \text{ all }t\ge 0. 
\end{equation}
One can also obtain \eqref{eq:conservation} by taking the dot-product of \eqref{eq:Motion}$_3$ by $\T I\cdot (\V \omega_1-\V \omega_R)$. 
 
\section{Weak solutions and their properties}\label{sec:weak}

Our investigation on the inertial motion about a fixed point of the system of two rigid bodies with a liquid-filled gap is carried out in a considerably large class of solutions to \eqref{eq:Motion} having finite kinetic energy.  

A {\em weak formulation} for the problem \eqref{eq:Motion}, can be found by dot-multiplying both sides of \eqref{eq:Motion}$_1$ by $\V \varphi\in\mathcal H^1_2(\mathcal V)$, integrating (by parts) the resulting equation over $\Ell\times (0,t)$, and using \eqref{eq:Motion}$_{3,4}$ together with \eqref{eq:xtimesomegatimesx} and \eqref{eq:omegatimesxtimesomegatimesx}. This leads to the following of problem: {\em find a solution $(\V{\tilde v},\V \Omega)$ to the following system of equations}
\begin{equation}\label{eq:weak}
\begin{aligned}
(\V{\tilde v}(t), \V \varphi)_B 
&+2\mu\int^t_0\int_{\mathcal V}\T D(\V{\tilde v})\dotdot \T D(\V \varphi)\;\d V\d \tau
+ b(\V \varphi)\cdot\int^t_0 [\V \Omega+b(\V{\tilde v})]\times \T I\cdot \V \Omega\;\d \tau 
\\
&+ \int^t_0\int_{\mathcal V}\tilde \rho[\V{\tilde v}\cdot \nabla \V{\tilde v}
+2(\V \Omega+b(\V{\tilde v}))\times \V{\tilde v}]\cdot \V \varphi\;\d V\d \tau=(\V{\tilde v}(0), \V \varphi)_B,
\\
& \qquad\qquad\qquad\qquad\qquad \qquad\text{ for all $\V \varphi\in\mathcal H^1_2(\mathcal V)$, }\text{and all  $t\in[0,\infty)$. }
\\
\T I\cdot\V{\Omega}(t)&+\int^t_0[\V \Omega+b(\V{\tilde v})]\times \T I\cdot \V \Omega\;\d \tau=\T I\cdot\V{\Omega}(0), \qquad\qquad\text{for all }t\in[0,\infty).
\end{aligned}
\end{equation}

\begin{definition}\label{def:weak}
The triple $(\V v,\V \omega_1,\V \omega)$ is a {\em weak solution} to \eqref{eq:Motion} if the following requirements are met. 
\begin{enumerate}
\item Consider the field $\V{\tilde v}$ in \eqref{eq:extension}. Then, 
\[
\V{\tilde v}\in C_w ([0,\infty);\mathcal H_2(\mathcal V))\cap L^\infty(0,\infty;\mathcal H_2(\mathcal V))\cap L^2(0,\infty;\mathcal H^1_2(\mathcal V)).
\]
\item The vector field $\V \Omega=\V \omega_1-b(\V{\tilde v})\in C^0([0,\infty))\cap C^1((0,\infty))$.  
\item $(\V{\tilde v},\V \Omega)$ satisfies \eqref{eq:weak}.
\item The following {\em strong energy inequality} holds: 
\begin{multline}\label{eq:strong_energy}
\mathcal E(\V{\tilde v}(t))+\V \Omega(t)\cdot \T I\cdot \V \Omega(t)+4\mu\int^t_s\norm{\T D( \V{\tilde v}(\tau))}^2_{L^2(\Ell)}\;\d \tau
\le\mathcal E(\V{\tilde v}(s))+\V \Omega(s)\cdot \T I\cdot 
\V \Omega(s),
\end{multline}
for all $t\ge s$ and a.a. $s\ge 0$ including $s=0$. 
\end{enumerate}
\end{definition}

From the previous definition, it immediately follows that the physical velocity fields $(\V v,\V \omega_1,\V \omega)$ enjoy the following properties 
\begin{equation}\label{eq:regularity_weak}
\begin{split}
&\V v\in C_w ([0,\infty);H(\Ell))\cap L^\infty(0,\infty;H(\Ell))\cap L^2(0,\infty;H(\Ell)\cap W^{1,2}(\Ell))
\\
&\V \omega_1\in C([0,\infty))\cap L^\infty(0,\infty),
\\
&\V \omega\in C([0,\infty))\cap L^\infty(0,\infty)\cap L^2(0,\infty),
\\
&\V v=\V 0\quad \text{on }\; \mathcal C, \qquad\V v=\V \omega\times \V x\quad\text{on }\;\mathcal S\text{ (in the trace sense).}   
\end{split}
\end{equation}
In particular, if $(\V v,\V \omega_1,\V \omega)$  is a weak solution, by \eqref{eq:strong_energy} together with \eqref{eq:kokr} and \eqref{eq:inner_product_w}, it follows that there exists a constant $c_0=c_0(\V v(0),\V \Omega(0),\V\omega(0))$ such that 
\[
\rho\norm{\V v}_{L^2(\Ell)}^2+\lambda|\V \omega|^2\le c^2_0,\;\qquad \quad \text{for all }\; t\ge 0. 
\]
Furthermore, up to redefining the above constant $c_0$, we also have 
\[\begin{aligned}
|\V \omega_R(t)|\le \rho\int_{\Ell}|\V x\times \V v|\; \d V+\lambda|\V \omega&|\le c_0\;&&\text{ for all }\; t\ge 0, 
\\
\V \Omega(t)\cdot \T I\cdot \V \Omega(t)=\V \omega_1(t)\cdot \T I\cdot \V \omega_1(t)-2\V \omega_R(t)\cdot \T I \cdot \V \omega_1(t)&\le c^2_0\;&&\text{ for all }\; t\ge 0.
\end{aligned}\]
Thus, for every $\varepsilon>0$, 
\[\begin{split}
\lambda_{\min}|\V \omega_1(t)|^2\le c^2_0+2\V \omega_R(t)\cdot \T I \cdot \V \omega_1(t)
&\le c^2_0+2\lambda_{\max}|\V \omega_R(t)|\;|\V \omega_1(t)|
\\
&\le c^2_0+\frac{\lambda_{\max}}{\varepsilon}|\V \omega_R(t)|^2+\lambda_{\max}\varepsilon|\V \omega_1(t)|^2.
\end{split}\]
Here, $\lambda_{\min}$ and $\lambda_{\max}$ denote the minimum and maximum eigenvalue of $\T I$, respectively. Choosing $\varepsilon:=\lambda_{\min}/(2\lambda_{max})$, we can conclude that 
\[
\frac 12 \lambda_{\min}|\V \omega_1(t)|^2\le c^2_0\left(1+2\frac{\lambda^2_{\max}}{\lambda_{\min}}\right)\;\qquad \quad \text{for all }\; t\ge 0. 
\]

\begin{remark}
Equations \eqref{eq:weak} together with \eqref{eq:strong_energy} represent the ``classical'' weak formulation ({\em \`a la Leray-Hopf}) for the problem of a rigid body having a cavity $\mathcal V$ completely filled by a viscous liquid with the varying density $\tilde \rho$ defined in  \eqref{eq:density}. However, setting $\V \omega_1=\V \Omega+\V \omega_R$ and using \eqref{eq:xtimesomegatimesx} and \eqref{eq:omegatimesxtimesomegatimesx}, one can immediately observe that the system of equations in \eqref{eq:weak} is the appropriate weak formulation obtained by testing \eqref{eq:Motion} with functions $\V \psi\in C^\infty(\Ell)$, $\diver \V \psi=0$ on $\Ell$ and satisfying the boundary conditions $\V \psi=\V 0$ on $\mathcal C$ and $\V \psi=\V \omega_\psi\times \V x$ on $\mathcal S$. In fact, for such test functions\footnote{Due to its regularity, we can extend $\V \psi$ by its boundary value on $\mathcal B_2$ and use it as test function in \eqref{eq:weak}. } $\V \psi$ and all  $t\in(0,\infty)$,  
\begin{equation}\label{eq:weak_v}
\begin{split}
&\int_{\Ell} \rho[\V v(t)+\V \omega_1(t)\times \V x] \cdot \V \psi\; \d V 
+\lambda\left[\V \omega(t)+\V \omega_1(t)+\int^t_0\V \omega_1\times \V \omega\;\d \tau\right]\cdot \V \omega_\psi
\\
&\quad\qquad +2\mu\int^t_0\int_{\Ell}\T D(\V v)\dotdot \T D(\V \psi)\;\d V\d \tau+ \int^t_0\int_{\Ell}\rho[\V v\cdot \nabla \V v
+2\V \omega_1\times \V v]\cdot \V \psi\;\d V\d \tau
\\
&\quad=\int_{\Ell} \rho[\V v(0)+\V \omega_1(0)\times \V x] \cdot \V \psi\; \d V
+\lambda[\V \omega(0)+\V \omega_1(0)]\cdot\V \omega_\psi,
\\
&\T I\cdot(\V \omega_1(t)-\V \omega_R(t))+\int^t_0\V \omega_1\times \T I\cdot \V \Omega\;\d \tau=\T I\cdot(\V \omega_1(0)-\V \omega_R(0)). 
\end{split}
\end{equation}
\end{remark}

\begin{remark}\label{re:strong_2} 
Assume that $\V{\tilde v}$ possesses enough regularity to allow differentiation with respect to time and integration by parts in \eqref{eq:weak}$_1$. Then 
\[
\V \omega_1=\V \Omega+b(\V{\tilde v})=\V \Omega+\V \omega_R\;\in C^1(0,\infty),
\]
and \eqref{eq:Motion}$_3$ is satisfied for a.a. $t\in (0,\infty)$. Moreover, the fields $\V v$ and $\V \omega\times \V x$ in \eqref{eq:extension} maintain the same regularity of $\V{\tilde v}$ on $\Ell$ and $\mathcal B_2$, respectively. 

By \eqref{eq:weak}, we find that $\V{\tilde v}$ also satisfies  
\begin{equation}\label{eq:weak_s}
\begin{split}
&(\frac{\partial \V{\tilde v}}{\partial t}+\V{\dot \omega}_1\times \V x+\V{\tilde v}\cdot \nabla \V{\tilde v}
+2\V{\omega}_1\times \V{\tilde v}, \V \varphi) +2\mu\int_{\mathcal V}\T D(\V{\tilde v})\dotdot\T D(\V \varphi)
=0
\end{split}
\end{equation}
for all $\V \varphi\in\mathcal H^1_2(\mathcal V)$ and all  $t\in(0,\infty)$. In particular, 
\[\begin{split}
\int_{\Ell}\left[\rho\left(\frac{\partial \V v}{\partial t}+\V{\dot \omega}_1\times \V x+\V{v}\cdot \nabla \V{v}
+2\V{\omega}_1\times \V{v}\right)-\mu\Delta \V v\right]\cdot \V \varphi=0 
\end{split}\]
for every $\V \varphi\in H(\Ell)\cap W^{1,2}_0(\Ell)$. Thus, there exists $\tilde p\in L^2(0,\infty;W^{1,2}(\Ell))$ such that 
\[
\rho\left(\frac{\partial \V v}{\partial t}+\V{\dot \omega}_1\times \V x+\V{v}\cdot \nabla \V{v}
+2\V{\omega}_1\times \V{v}\right)-\mu\Delta \V v=\nabla \tilde p\qquad \text{a.e. in }\Ell\times (0,\infty). 
\]
Set 
\[
p:=\tilde p-\frac \rho2 |\V{\omega}_1\times \V x|^2\qquad \text{in }\Ell, 
\] 
then one immediately notices that equations \eqref{eq:Motion}$_{1,2,5,6}$ are satisfied almost everywhere in space-time. Dot-multiplying \eqref{eq:Motion}$_1$ by $\V\varphi \in \mathcal H^1_2(\mathcal V)$ such that $\V\omega_\varphi=\V e_i$, $i=1,\,2,\,3$, and integrating the resulting equation over $\Ell$ we find 
\begin{multline*}
\int_{\Ell}\rho\left[\frac{\partial \V v}{\partial t}+\V{\dot \omega}_1\times \V x+
\V v\cdot \nabla \V v+2\V \omega_1\times \V v\right]\cdot \V \varphi
=\int_{\mathcal S}(\V x\times \T T\cdot \V n)\cdot \V e_i-2\mu\int_{\Ell} \T D(\V v)\dotdot \T D(\V \varphi). 
\end{multline*}Using \eqref{eq:xtimesomegatimesx} and \eqref{eq:omegatimesxtimesomegatimesx}, the latter displayed equation is equivalent to the following one: 
\begin{multline*}
(\frac{\partial \V{\tilde v}}{\partial t}+\V{\dot \omega}_1\times \V x+\V{\tilde v}\cdot \nabla \V{\tilde v}
+2\V{\omega}_1\times \V{\tilde v}, \V \varphi) +2\mu\int_{\mathcal V}\T D(\V{\tilde v})\dotdot\T D(\V \varphi)
\\
-\lambda(\V{\dot \omega}+\V{\dot \omega}_1+\V \omega_1\times \V \omega)\cdot \V e_i=\int_{\mathcal S}(\V x\times \T T\cdot \V n)\cdot \V e_i. 
\end{multline*}
By \eqref{eq:weak_s}, we can then conclude that 
\[
\lambda(\V{\dot \omega}+\V{\dot \omega}_1+\V \omega_1\times \V \omega)\cdot \V e_i
=-\int_{\mathcal S}(\V x\times \T T\cdot \V n)\cdot \V e_i, 
\]
for all $i=1,2,3$, and this proves that also \eqref{eq:Motion}$_4$ is satisfied. 
\end{remark}

The proof of the existence of weak solutions will be accomplished by using the Galerkin method together with a suitable approximation of the liquid velocity in $\mathcal H_2(\mathcal V)$.  To this aim, we will prove  the existence of a special basis of $\mathcal H_2(\mathcal V)$ and of a special basis of $\mathcal H^2_2(\mathcal V)$. We start by noticing that, taking \eqref{eq:norm_1q} with $q=2$, the norm $\norm{\cdot}_{1,2}$ is induced by the following inner product 
\begin{equation}\label{eq:inner_product_w_1}
(\V v,\V w)_1=(\V v,\V w)+2\mu\int_{\Ell}\T D(\V v)\dotdot \T D(\V w)\; \d V,
\end{equation}
and the latter makes $\mathcal H^1_2(\mathcal V)$ a Hilbert space.  

Consider the bilinear form $a:\; \mathcal H^1_2(\mathcal V)\times \mathcal H^1_2(\mathcal V)\to \R$ defined as follows 
\begin{equation}\label{eq:a}
a(\V v,\V w):=2\mu\int_{\Ell}\T D(\V v)\dotdot\T D(\V w).
\end{equation}
By \eqref{eq:norm_1q} and \eqref{eq:korn_q} with $q=2$, $a(\cdot,\cdot)$ is a continuous and coercive bilinear form in $\mathcal H^1_2(\mathcal V)$. Thus, by Lax-Milgram Theorem, for every $\V f\in \mathcal H_2(\mathcal V)$ there exists a unique solution $\V w\in \mathcal H^1_2(\mathcal V)$ to the variational problem 
\begin{equation}\label{eq:variational_stokes}
a(\V w,\V \varphi)=( \V f,\V \varphi),\qquad \text{ for all }\V \varphi\in \mathcal H^1_2(\mathcal V),
\end{equation}
where the inner product $(\cdot,\cdot )$ has been defined in \eqref{eq:inner_product_w}. In other words, $\V w$ is a generalized solution (with respect to the inner product \eqref{eq:inner_product_w}) to the problem 
\begin{equation}\label{eq:stokes0}
\begin{split}
&\left.\begin{split}
&-\frac{1}{\tilde \rho}\diver \T T(\V{\tilde v},p)=\V g
\\
&\diver \V{\tilde v}=0
\end{split}\quad\right\}\quad \text{in }\mathcal V
\\
&\ \  \V{\tilde v}=\V 0\qquad \text{on }\mathcal C,
\end{split}
\end{equation}
where $\V g\in L^2(\mathcal V)$ is such that $\V f=\mathcal P_2 \V g$. \footnote{We recall that $\mathcal P_2$ is the orthogonal projection of $L^2_R(\mathcal V)$ onto $\mathcal H_2(\mathcal V)$ with respect to the inner product $(\cdot,\cdot)$, defined in \eqref{eq:inner_product_w} (see Section \ref{sec:functional}). } 

With an argument similar to the one that leads to the classical estimates for the Stokes problem (see \cite[Theorem IV.6.1]{Ga}), one can further show that $\V w\in \mathcal H^2_2(\mathcal V)$, and there exists a unique (up to a constant) pressure field $q\in W^{1,2}(\mathcal V)$ such that equations \eqref{eq:stokes0}$_{1,2}$ are satisfied almost everywhere on $\mathcal V$. Moreover, $(\V w, q)$ satisfies the following estimates 
\begin{equation}\label{eq:stokes_estimate}
\norm{\V w}_{2,2}+\norm{q}_{W^{1,2}(\mathcal V)}\le c\norm{\V g}_2, 
\end{equation}
with $c=c(\mu,\rho,\lambda,R,\mathcal V)$ a positive constant. 

Consider the linear operator 
\[
A:\; \V u\in \mathcal H^2_2(\mathcal V)\mapsto A\V u:=-\nu \mathcal P(\Delta \V u)\in \mathcal H_2(\mathcal V), 
\]
where $\nu:=\mu/\rho$ is the liquid coefficient of kinematic viscosity. An integration by parts implies that $a(\V u,\V w)=(A\V u,\V w)$ for all $\V u,\V w\in  \mathcal H^2_2(\mathcal V)$. Thus, $A$ is a symmetric operator. Moreover, $A$ is invertible and closed. In fact, the inverse is defined by the operator 
\[
A^{-1}:\;\V f\in\mathcal H_2(\mathcal V)\mapsto A^{-1}\V f=\V{\tilde w}\in \mathcal H^1_2(\mathcal V), 
\]
the unique solution to \eqref{eq:variational_stokes}, and $A^{-1}$ is bounded  because of \eqref{eq:stokes_estimate}. Therefore, $A$ and $A^{-1}$ are self-adjoint.  In addition, thanks to the estimate \eqref{eq:stokes_estimate}, we have the following lemma. 

\begin{lemma}
There exists a positive constant $c$ such that  
\begin{equation}\label{eq:Pdelta}
\norm{\V w}_{2,2}\le c\nu\norm{\mathcal P(\Delta \V w)}_2\qquad \text{for all }
\V w\in \mathcal H^2_2(\mathcal V). 
\end{equation}
\end{lemma}

Let us consider the following inner product in $\mathcal H^2_2(\mathcal V)$
\begin{equation}\label{eq:inner_product_w_2}
(\V u,\V w)_2:=(A\V u,A\V w),\qquad\text{for all }\;\V u,\V w\in \mathcal H^2_2(\mathcal V).
\end{equation}
By \eqref{eq:Pdelta}, the associated norm is equivalent to $\norm{\cdot}_{2,2}$. We are now ready to prove the existence of a special basis.

\begin{theorem}\label{th:basis}
The spectral problem 
\begin{equation}\label{spectral}
(\V u,\V \varphi)_2=\lambda (\V u,\V \varphi)_B\qquad \text{for all}\; \V \varphi\in\mathcal H^2_2(\mathcal V)
\end{equation}
admits a denumerable number of positive eigenvalues $\{\lambda_n\}_{n\in \N}$ clustering at $+\infty$. The corresponding eigenfunctions $\{\V w_n\}_{n\in \N}$ belong to $\mathcal H^2_2(\mathcal V)$and  form an orthonormal basis in $\mathcal H_2(\mathcal V)$ with respect to the inner product $( \cdot,\cdot)_B$ defined in \eqref{eq:inner_product_B}. 

Furthermore,   $\{\V w_n/\sqrt{\lambda_n}\}_{n\in \N}$ forms an orthonormal basis in $H^2_2(\mathcal V)$ 
with respect to the inner product $(\cdot,\cdot)_2$ defined in \eqref{eq:inner_product_w_2}. 
\end{theorem}

\begin{proof}
By Lemma \ref{le:kokr} and Lax-Milgram Theorem, for every $\V f\in\mathcal H_2(\mathcal V)$ there exists a unique solution to the problem
\begin{equation}\label{eq:lax}
(\V u,\V \varphi)_2= (\V f,\V \varphi)_B\qquad \text{for all}\; \V \varphi\in\mathcal H^2_2(\mathcal V). 
\end{equation}
Consider the operator $S_0:\V f\in\mathcal H_2(\mathcal V)\mapsto S_0\V f:=\V u\in\mathcal H^2_2(\mathcal V)$ the unique solution to \eqref{eq:lax}. By Lemma \ref{lem:embeddingk}, the injection $J:\mathcal H^2_2(\mathcal V)\to \mathcal H_2(\mathcal V)$ is compact. Thus, the operator $S:=J\circ S_0:\V f\in\mathcal H_2(\mathcal V)\to S\V f\in\mathcal H_2(\mathcal V)$ is also compact. Moreover, $S$ is symmetric with respect to the inner product $( \cdot,\cdot)_B$ defined in \eqref{eq:inner_product_B}. In fact, for every $\V f_1$ and $\V f_2\in\mathcal H_2(\mathcal V)$, we know that there exist unique $\V u_1$ and $\V u_2\in\mathcal H^2_2(\mathcal V)$ solutions to \eqref{eq:lax} with $\V f$ replaced by $\V f_1$ and $\V f_2$, respectively. So, $S\V f_1=\V u_1$ and $S\V f_2=\V u_2$, and 
\begin{multline*}
( S\V f_1,\V f_2)_B=(\V u_1,\V f_2)_B=( \V f_2,\V u_1)_B=(\V u_2,\V u_1)_2
=(\V u_1,\V u_2)_2=( \V f_1,\V u_2)_B=( \V f_1,S\V f_2)_B.
\end{multline*}
In addition, if $\V f_1=\V f_2\equiv\V f$, then $\V u_1=\V u_2\equiv\V u$ and $( S\V f,\V f)_B=(\V u,\V u)_2$. Thus, $S$ is also a positive definite operator. Finally, $S$ is self-adjoint. To prove the latter, we notice that $S$ is a compact perturbation of the identity, and $-1$ is not an eigenvalue of $S$. Thus, $Range(S)=\mathcal H_2(\mathcal V)$ (\cite[Theorem 1, Section 5, Chapter X]{Yosida}). Since $Range(S)=\mathcal H_2(\mathcal V)$ and $S$ is symmetric, then $S$ is self-adjoint (\cite[Corollary to Theorem 1, Section 3, Chapter VII ]{Yosida}). By the Hilbert-Schmidt Theorem, $(\mathcal H_2(\mathcal V),(\cdot,\cdot)_B)$ admits an orthonormal basis of eigenfunctions $\{\V w_n\}_{n\in \N}$ of $S$ with corresponding positive eigenvalues $\{\nu_n\}_{n\in N}$ converging to $0$ as $n\to\infty$. 

Let us denote $\lambda_n:=\nu_n^{-1}>0$ for every $n\in \N$. So, $\{\lambda_n\}_{n\in \N}$ forms a sequence of eigenvalues of the problem \eqref{eq:lax} clustering at infinity as $n\to \infty$ and with corresponding eigenfunctions $\{\V w_n\}_{n\in \N}$. Indeed, by the definition of $S$, we find that $\V w_n\in\mathcal H^2_2(\mathcal V)$ and 
\[
\nu_n(\V w_n,\V \varphi)_2=( S\V w_n,\V \varphi)_2=( \V w_n,\V \varphi)_B,\quad\text{for every }\V \varphi\in\mathcal H^2_2(\mathcal V),\;n\in \N. 
\]

Finally, $\{\V w_n/\sqrt{\lambda_n}\}_{n\in \N}$ forms an orthonormal basis in $H^2_2(\mathcal V)$ with respect to the inner product $(\cdot,\cdot)_2$ defined in \eqref{eq:inner_product_w_2}. To see this, consider $\V u\in \mathcal H^2_2(\mathcal V)$ be such that $(\V w_n,\V u)_2=0$ fo every $n\in \N$. Then, 
\[
0=\nu_n\left(\V w_n,\V u\right)_2=(S\V w_n,\V u)_2=( \V w_n,\V u)_B
\] 
for every $n\in \N$, and this implies that  $\V u=0$ since $\{\V w_n\}_{n\in \N}$ forms a basis in $\mathcal H_2(\mathcal V)$ endowed with the inner product $(\cdot,\cdot)_B$. Therefore,  $\{\V w_n/\sqrt{\lambda_n}\}_{n\in \N}$ is a basis of $\mathcal H^2_2(\mathcal V)$. Furthermore, 
\[\begin{split}
(\frac{\V w_n}{\sqrt{\lambda_n}},\frac{\V w_m}{\sqrt{\lambda_m}})_2&=\frac{1}{\sqrt{\lambda_n}}
\frac{1}{\sqrt{\lambda_m}}\left(\V w_n,\V w_m\right)_2
=\frac{\lambda_n}{\sqrt{\lambda_n}\sqrt{\lambda_m}}\left(S\V w_n,\V w_m\right)_2
\\
&=\frac{\lambda_n}{\sqrt{\lambda_n}\sqrt{\lambda_m}}( \V w_n,\V w_m)_B
=\frac{\lambda_n}{\sqrt{\lambda_n}\sqrt{\lambda_m}}\delta_{nm}\qquad\text{for all }n,m\in \N. 
\end{split}\]
\end{proof}

We are now in position to prove the following result about the existence of weak solutions to \eqref{eq:Motion}. 

\begin{theorem}\label{th:weak}
For every $\V v_0\in H(\Ell)$, $\V \omega_{10},\; \V \omega_{0}\in \R^3$ such that $\V v_0=\V \omega_0\times \V x$ on $\mathcal S$, there exists at least one weak solution to \eqref{eq:Motion} such that 
\begin{enumerate}
\item $\lim_{t\to 0^+}\norm{\V{v}(t)-\V{v}_0}_2=\lim_{t\to 0^+}|\V \omega_1(t)-\V \omega_{10}|=\lim_{t\to 0^+}|\V \omega(t)-\V \omega_{0}|=0$.
\item The following decays hold
\begin{equation}\label{eq:decay0}
\lim_{t\to\infty}\norm{\V v}_{L^2(\Ell)}=0\quad \text{and}\quad\lim_{t\to \infty}|\V \omega(t)|=0. 
\end{equation}
In particular, if $\lambda_1=\lambda_2=\lambda_3$, then the rate of the previous decays is exponential. 
\item Equation \eqref{eq:conservation} holds. 
\end{enumerate}
\end{theorem}

\begin{proof}
Consider the basis of $\mathcal H_2(\mathcal V)$ constructed in Theorem \ref{th:basis}. We look for ``approximate'' solutions 
\begin{equation}\label{eq:approximations}
\V{\tilde v}_n(\V x,t)=\sum^n_{p=1}c_{np}(t)\V w_p(\V x),\qquad \V\Omega_{n}(t)=\sum^3_{i=1}\hat c_{ni}(t)\V e_i
\end{equation}
satisfying \eqref{eq:weak}$_1$ with $\V \varphi=\V w_r$, and \eqref{eq:weak}$_2$. Set 
\[
\V{\tilde v}_0:=\left\{\begin{split}
\V v_0\qquad &\text{in }\Ell,
\\
\V \omega_0\times \V x\quad &\text{in }\mathcal B_2. 
\end{split}\right.
\]
Then, $\V{\tilde v}_0\in \mathcal H_2(\mathcal V)$. Moreover, set $\V \Omega_0:=\V \omega_{10}+b(\V{\tilde v}_0)\in \R^3$. 

Let $\V{\tilde v}_{0n}$ denote the projection of $\V{\tilde v}_0$ on the $span\{\V w_1,\dots,\V w_n\}$. Replacing \eqref{eq:approximations} in \eqref{eq:weak}, we find that $(c_{nr},\hat c_{nk})_{r=1,\dots,n,\;k=1,2,3}$ satisfy the following system of $(n+3)\times(n+3)$ first order initial value problems 
\begin{equation}\label{eq:coefficients}\begin{aligned}
&\left.\begin{aligned}
&\dot c_{nr}(t)+2\mu \sum^n_{p=1}a_{pr}c_{np}(t)+\sum^n_{p=1}\sum^n_{q=1}b_{pqr}c_{np}(t)c_{nq}(t)
\\
&\qquad\qquad +\sum^3_{i=1}\sum^n_{p=1}d_{ipr}\hat c_{ni}(t)c_{np}(t)
+\sum^3_{i=1}\sum^3_{j=1}f_{ijr}\hat c_{ni}(t)\hat c_{nj}(t)=0
\\
&c_{nr}(0)=( \V{\tilde v}_{0n},\V w_r)_B
\end{aligned}\right\}&&\text{ for }r=1,\dots,n,
\\
&\left.\begin{split}
&\ell_k\dot{\hat c}_{nk}(t)+\sum^3_{i=1}\sum^3_{j=1}g_{ijk}\hat c_{ni}(t)\hat c_{nj}(t)
+\sum^n_{p=1}\sum^3_{j=1}h_{pjk}c_{np}(t)\hat c_{nj}(t)=0
\\
&\hat c_{nk}(0)=\V\Omega_{0}\cdot \V e_k
\end{split}\right\}&&\text{ for }k=1,2,3,
\end{aligned}\end{equation} 
where the (constant) coefficients are: $\ell_k:=\V e_k\cdot\T I\cdot\V e_k>0$, 
\[\begin{split}
&a_{pr}:=2\mu\int_{\mathcal V}\T D(\V w_p)\dotdot\T D(\V w_r)\;\d V,
\\
&b_{pqr}:=\int_{\mathcal V}\tilde \rho\left[\V w_p\cdot \nabla\V{\tilde w}_q 
-2\left(\T I^{-1}\cdot\int_{\mathcal V}\tilde \rho \V x\times \V w_p\;\d V\right)\times \V{\tilde w}_q\right]\cdot \V w_r\;\d V,
\\
&d_{ipr}:=2\int_{\mathcal V}\tilde \rho (\V e_i\times\V w_p)\cdot\V w_r
-\V e_i\cdot \T I\cdot\left[\left(\T I^{-1}\cdot\int_{\mathcal V}\tilde \rho \V x\times \V w_p\;\d V\right)\times\left(\T I^{-1}\cdot\int_{\mathcal V}\tilde \rho \V x\times \V w_r\;\d V\right)\right]\d V,
\\
&f_{ijr}:=\V e_i\cdot\left[\left(\T I^{-1}\cdot\int_{\mathcal V}\tilde \rho \V x\times \V w_r\;\d V\right)\times \T I\cdot \V e_j\right], 
\quad g_{ijk}:=\V e_k\cdot (\V e_i\times \T I\cdot \V e_j),
\\
&h_{pjk}:=-\V e_k\cdot \left[\left(\T I^{-1}\cdot\int_{\mathcal V}\tilde \rho \V x\times \V w_p\;\d V\right)\times \T I\cdot \V e_j\right]. 
\end{split}\]
By the classical theory of ordinary differential equations, the initial value problem \eqref{eq:coefficients} admits a unique solution $(c_{nr},\hat c_{nk})_{r=1,\dots,n,k=1,2,3}$ defined in some interval $[0,T_n)$ with $T_n>0$. Actually, $T_n=+\infty$ for all $n\in \N$. In fact, the approximate solutions satisfy the following system of equations 
\begin{equation}\label{eq:weak_n}
\begin{aligned}
&(\frac{\d \V{\tilde v}_n}{\d t}, \V w_r)_B 
+2\mu\int_{\mathcal V}\T D(\V{\tilde v}_n)\dotdot \T D(\V w_r)\;\d V
+ b(\V w_r)\cdot[(\V \Omega_n+b(\V{\tilde v}_n))\times \T I\cdot \V \Omega_n]
\\
&\qquad\qquad\qquad\quad + \int_{\mathcal V}\tilde \rho[\V{\tilde v}_n\cdot \nabla \V{\tilde v}_n
+2(\V \Omega_n+b(\V{\tilde v}_n))\times \V{\tilde v}_n]\cdot \V w_r\;\d V=0, 
&&\text{ for all }r,n\in \N,
\\
&\T I\cdot\V{\dot \Omega}_n+[\V \Omega_n+b(\V{\tilde v}_n)]\times \T I\cdot \V \Omega_n=0,&&\text{ for all }n\in \N, 
\end{aligned}
\end{equation}
and the energy equality
\begin{equation}\label{eq:energy_n}
\frac 12 \frac{\d }{\d t}\left[\mathcal E(\V{\tilde v}_n)+\V\Omega_{n}\cdot \T I \cdot \V \Omega_{n}\right]+2\mu\norm{\T D(\V{\tilde v}_n)}^2_{L^2(\Ell)}=0\quad \text{ in }(0,T_n),\text{ for all }n\in \N. 
\end{equation}
The latter equality is obtained by multiplying \eqref{eq:coefficients}$_1$ by $c_{nr}$ and summing over $r=1,\dots,n$, by multiplying \eqref{eq:coefficients}$_3$ by $\hat c_{nk}$ and summing over $k=1,2,3$, and then adding  the resulting equations. Integrating \eqref{eq:energy_n}  in $[0,t]$, $t<T_n$, and using 
\eqref{eq:kokr}, we find that 
\begin{equation}\label{eq:apriori2}
c\norm{\V{\tilde v}_n(t)}_2^2+\V\Omega_{n}(t)\cdot \T I \cdot \V \Omega_{n}(t)+2\mu\int_0^t\norm{\T D(\V{\tilde v}_n)}^2_{L^2(\Ell)}\;\d \tau\le \norm{\V{\tilde v}_0}^2_2+\V\Omega_{0}\cdot \T I \cdot \V \Omega_{0},
\end{equation}
for all $t\in[0,T_n)$. Since the right-hand side does not depend on $n$ and $t$, necessarily $T_n=+\infty$ by the standard continuation theorem for ordinary differential equations. 
Moreover, the sequence $\{(\V{\tilde v}_n,\V \Omega_n)\}_{n\in \N}$ enjoys the following properties. 
\begin{enumerate}
\item[(a)] By \eqref{eq:apriori2}, $\{\V{\tilde v}_n\}_{n\in \N}$ is uniformly bounded in $L^\infty(0,\infty;\mathcal H_2(\mathcal V))$.
\item[(b)]  $\{\V{\tilde v}_n\}_{n\in \N}$ is uniformly bounded also in $L^2(0,\infty;\mathcal H^1_2(\mathcal V))$ by \eqref{eq:apriori2} and \eqref{eq:sobolev_korn}. 
\item[(c)] $\{\V \Omega_n\}_{n\in \N}$ is uniformly bounded in $C^0([0,\infty))\cap C^1(0,\infty)$, by \eqref{eq:weak_n}$_2$ and \eqref{eq:apriori2}. 
\item[(d)]$\{\d\V{\tilde v}_n/\d t\}_{n\in \N}$ is uniformly bounded in $L^2(0,T;(\mathcal H^2_2(\mathcal V))')$ for every $T>0$.  To show this, let $\mathbb P_n$ be the orthogonal projection of $\mathcal H^2_2(\mathcal V)$ onto $span\{\V w_1/\sqrt{\lambda_1},\dots,\V w_n/\sqrt{\lambda_n}\}$. By Theorem \ref{th:basis}, for every $\V w\in\mathcal H^2_2(\mathcal V)$ one has 
\begin{equation}\label{eq:mathbbPn}
\V w=\sum^\infty_{\ell=0}(\V w,\V w_\ell)_2\V w_\ell\qquad \text{and}\qquad\norm{\mathbb P_n\V w}_{2,2}\le \norm{\V w}_{2,2},\quad\text{for all }n\in \N.
\end{equation}
For every $\V w\in\mathcal H^2_2(\mathcal V)$, 
\[\begin{split}
(\frac{\d \V{\tilde v}_n}{\d t}, \V w)_B&=(\frac{\d \V{\tilde v}_n}{\d t}, \mathbb P_n\V w)_B
=-2\mu\int_{\mathcal V}\T D(\V{\tilde v}_n)\dotdot \T D(\mathbb P_n\V w)\;\d V
\\
&\quad - \int_{\mathcal V}\tilde \rho[\V{\tilde v}_n\cdot \nabla \V{\tilde v}_n
+2(\V \Omega_n+b(\V{\tilde v}_n))\times \V{\tilde v}_n]\cdot (\mathbb P_n\V w)\;\d V
\\
&\quad-b(\mathbb P_n\V w)\cdot[(\V \Omega_n+b(\V{\tilde v}_n))\times \T I\cdot \V \Omega_n]\quad 
\text{ for all }n\in \N. 
\end{split}\]
We recall the following classical estimates that can be obtained using an integration by parts together with H\"older inequality, \eqref{eq:korn_q} and \eqref{eq:sobolev_korn}. For every $\V u_1, \V u_2\in \mathcal H^1_2(\mathcal V)$ and $\V z\in \mathcal H^2_2(\mathcal V)$ one has 
\begin{multline}\label{eq:nonlinear}
\left|\int_{\mathcal V}\tilde \rho(\V u_1\cdot\nabla \V u_2)\cdot \V z\; \d V\right|= \left|\int_{\mathcal V}\tilde \rho(\V u_1\cdot\nabla\V z)\cdot \V u_2\; \d V\right|
\\
\le \norm{\V u_1}_6\norm{\nabla \V z}_3\norm{\V u_2}_2
\le c\norm{\T D(\V u_1)}_{L^2(\Ell)}\norm{\V z}_{2,2}\norm{\V u_2}_2. 
\end{multline}
Using again H\"older inequality, \eqref{eq:nonlinear} and \eqref{eq:apriori2}, we find that 
\begin{multline*}
\left|(\frac{\d \V{\tilde v}_n}{\d t}, \V w)_B\right|=\left|(\frac{\d \V{\tilde v}_n}{\d t}, \mathbb P_n\V w)_B\right|
\le c_1\norm{\T D(\V{\tilde v}_n)}_{L^2(\Ell)} \norm{\V w}_{2,2}
\\
+c_2\norm{ \T D(\V{\tilde v}_n)}_{L^2(\Ell)}\norm{\V w}_{2,2}\norm{\V{\tilde v}_n}_2
+c_3\norm{\V{\tilde v}_n}_2\norm{\V w}_{2,2}
+c_4|\V \Omega_n|\norm{\V w}_{2,2}. 
\end{multline*}
Since the previous estimates hold for every $\V w\in \mathcal H^2_2(\mathcal V)$ and $\mathcal H_2(\mathcal V)\hookrightarrow (\mathcal H^2_2(\mathcal V))'$ , by properties (a), (b) and (c), we can conclude that the sequence $\{\d\V{\tilde v}_n/\d t\}_{n\in \N}$ belongs to a bounded set of $L^2(0,T;(\mathcal H^2_2(\mathcal V))')$ for every $T>0$. 
\end{enumerate}
Properties (b) and (d) imply that the sequence $\{\V{\tilde v}_n\}_{n\in \N}$ remains in a bounded set of the following space 
\[
\{\V u\in L^2(0,T;\mathcal H^1_2(\mathcal V)):\; \d \V u/\d t\in L^2(0,T;\mathcal (H^2_2(\mathcal V))')  \}.
\]
Moreover,  $\mathcal H^1_2(\mathcal V))\hookrightarrow \mathcal H_2(\mathcal V)\hookrightarrow (H^2_2(\mathcal V))'$, with the first embedding being compact (Lemma \ref{lem:embedding1}). Taking into account all these features and properties (a)-(d), we can claim the existence of functions 
\[\begin{split}
&\V{\tilde v}\in L^\infty(0,\infty;\mathcal H_2(\mathcal V))\cap L^2(0,\infty;\mathcal H^1_2(\mathcal V)),
\\
&\V \Omega\in C^0([0,\infty))\cap C^1(0,\infty),
\end{split}\]
and subsequences, again denoted by $\{\V{\tilde v}_n\}_{n\in\N}$ and $\{\V \Omega_n\}_{n\in\N}$, such that 
\begin{equation}\label{eq:convergence}\begin{split}
&\lim_{n\to \infty}\V{\tilde v}_n=\V{\tilde v}\quad\text{weakly$-*$ in }L^\infty(0,\infty;\mathcal H_2(\mathcal V)),
\\
&\lim_{n\to \infty}\V{\tilde v}_n=\V{\tilde v}\quad\text{weakly in }L^2(0,\infty;\mathcal H^1_2(\mathcal V)),
\\
&\lim_{n\to \infty}\V \Omega_n=\V \Omega\quad\text{uniformly in every closed interval }J\subset[0,\infty),
\\
&\lim_{n\to \infty}\V{\tilde v}_n=\V{\tilde v}\quad\text{strongly in }L^2(0,T;\mathcal H_2(\mathcal V))\quad \text{for every }T>0.
\end{split}\end{equation}
The latter convergence is a consequence of properties (b) and (d), and of the Aubin-Lions compactness lemma (see \cite[Theorem 2.1, Chapter III]{Temam}).  

To conclude the proof of the theorem, we need to show that the couple $(\V{\tilde v},\V \Omega)$ satisfies \eqref{eq:weak}. In other words, we need to pass to the limit as $n\to \infty$ in the following equation obtained from \eqref{eq:weak_n}, after an integration with respect to time:  
\begin{equation}\label{eq:weak_nr}
\begin{split}
(\V{\tilde v}_n(t),\V \varphi)_B-(\V{\tilde v}_n(0),\V \varphi)_B 
&+2\mu\int^t_0\int_{\mathcal V}\T D(\V{\tilde v}_n)\dotdot \T D(\V \varphi)\;\d V\d \tau
\\
&+ \int^t_0\int_{\mathcal V}\tilde \rho[\V{\tilde v}_n\cdot \nabla \V{\tilde v}_n
+2(\V \Omega_n+b(\V{\tilde v}_n))\times \V{\tilde v}_n]\cdot\V \varphi\;\d V\d \tau
\\
&+ b(\V \varphi)\cdot\int^t_0 [\V \Omega_n+b(\V{\tilde v}_n)]\times \T I\cdot \V \Omega_n\;\d \tau=0,
\\
\T I\cdot\V{\Omega}_n(t)-\T I\cdot\V{\Omega}_n(0)&+\int^t_0[\V \Omega_n+b(\V{\tilde v}_n)]\times \T I\cdot \V \Omega_n\;\d \tau=0, \text{ for all }t\in[0,\infty).
\end{split}
\end{equation} 
Thanks to \eqref{eq:convergence}, the convergence of both linear and nonlinear terms in the above equations follows from standard arguments. We have then shown that, for every $T>0$, the couple $(\V{\tilde v},\V \Omega)$ satisfies \eqref{eq:weak} for every $\V \varphi\in \mathcal H^2_2(\mathcal V)$ and all $t\in [0,T)$. Since $\mathcal H^2_2(\mathcal V)$ is dense in $\mathcal H^1_2(\mathcal V)$, \eqref{eq:weak}$_1$ is also satisfied for every $\V \varphi\in \mathcal H^1_2(\mathcal V)$. Moreover, $\V{\tilde v}\in C_w([0,T);\mathcal H_2(\mathcal V))$ since it satisfies \eqref{eq:weak} in $[0,T)$ for every $T>0$. In fact, from the weak formulation, one can easily show that if $t_0\in [0,T)$, then for every $\varepsilon>0$ there exists $\delta=\delta(\varepsilon)>0$ such that for every $t\in (t_0-\delta,t_0+\delta)$:
\[
|(\V{\tilde v}(t)-\V{\tilde v}(t_0),\V \varphi)_B|<\varepsilon,\qquad \text{for all }\;\V \varphi\in\mathcal H^1_2(\mathcal V). 
\]
By the density of $\mathcal H^1_2(\mathcal V)$ in $\mathcal H_2(\mathcal V)$, the latter property continues to hold for every $\V \varphi\in\mathcal H_2(\mathcal V)$. 
In addition, taking the limit as $n\to \infty$ in \eqref{eq:apriori2} and using \eqref{eq:convergence}$_{2,3,4}$ with $\V{\tilde v}\in C_w([0,T);\mathcal H_2(\mathcal V))$, we can conclude that $(\V{\tilde v},\V \Omega)$ satisfies the strong energy inequality \eqref{eq:strong_energy}. 

Let us prove properties {\em 1.} to {\em 3.} in the statement. Let $\V \omega_1:=\V \Omega+b(\V{\tilde v})$ and recall that $\V{\tilde v}$ has the following representation
\[
\V{\tilde v}=\left\{\begin{split}
\V v\qquad &\text{in }\;\Ell,
\\
\V \omega\times \V x\;\quad &\text{in }\,\mathcal B_2.
\end{split}\right.
\]
Then, $(\V v,\V \omega_1,\V \omega)$ satisfy \eqref{eq:regularity_weak}. 

Recall \eqref{eq:energy_d} and \eqref{eq:kokr}, thus property {\em 1.} immediately follows from the strong energy inequality \eqref{eq:strong_energy} and the lower semicontinuity at zero of the map: $t\to \norm{\V v(t)}_2^2$.  

For what concerns the decays stated in property {\em 2.}, by \eqref{eq:strong_energy} and \eqref{eq:dissipation}, for all $t\ge s$ and a.a. $s\ge 0$ including $s=0$, we find that 
\[
\mathcal E(\V{\tilde v}(t))+C\mu\int^t_s\mathcal E(\V{\tilde v}(\tau))\;\d \tau\le\mathcal E(\V{\tilde v}(s))+G(t,s),
\]
where $G(t,s):=\V \Omega(t)\cdot \T I\cdot \V \Omega(t)-\V \Omega(s)\cdot \T I\cdot \V \Omega(s)$. By \eqref{eq:weak}$_2$, \eqref{eq:strong_energy} with $s=0$ and H\"older inequality, we find that 
\[
G(t,s)= 2\int^t_s\V \Omega\cdot[b(\V{\tilde v})\times \T I\cdot \V \Omega]\;\d \tau\le c_1\int^t_s F(\tau)\; \d \tau
\]
where $c_1$ is a positive constant (independent of time) and $F(t):=\norm{\V{\tilde v}(t)}_2$. Hence, \eqref{eq:decay0} 
follows by Lemma \ref{lem:gronwall1}. In particular, if $\lambda_1=\lambda_2=\lambda_3$, then $\V \Omega\cdot[b(\V{\tilde v})\times \T I\cdot \V \Omega]=0$, and also the exponential decay follows. 

Finally, we obtain \eqref{eq:conservation} from \eqref{eq:weak}$_2$ by dot-multiplying it by $\T I\cdot\V \Omega$ and 
recalling that $\V \Omega=\V \omega_1-~b(\V{\tilde v})$. 
\end{proof}

Due to the coupling with the Navier-Stokes equations, also for the problem at hand, it is an open problem whether weak solutions constructed in Theorem \ref{th:weak} continuously depend upon the initial data, and are in particular unique. Nevertheless, such property holds for any weak solution possessing a further regularity, as for the classical Navier-Stokes case.

\begin{theorem}\label{th:continuous_dependence}
Consider two weak solutions $(\V v,\V \omega_1,\V \omega)$ and $(\V v^*,\V \omega_1^*,\V \omega^*)$ to \eqref{eq:Motion} corresponding to initial data $(\V v_0,\V \omega_{10},\V \omega_0)$ and $(\V v_0^*,\V \omega_{10}^*,\V \omega_0^*)$, respectively. Suppose that there exists a time $T>0$ such that 
\begin{equation}\label{eq:serrin}
\V v^*\in L^p(0,T;L^q(\Ell)), \qquad \frac 2p+\frac 3q=1,\quad \text{for some }\; q>3. 
\end{equation}
Then, the following properties hold. 
\begin{itemize}
\item[a)] There exists a positive constant $c$ depending only on $\norm{\V v^*}_{L^\infty(0,T;L^2(\Ell))}$, $\norm{\V v^*}_{L^p(0,T;L^q(\Ell))}$, $\max_{t\in[0,T]}|\V \omega_1^*(t)|$ and $\max_{t\in[0,T]}|\V \omega^*(t)|$ such that 
\begin{multline*}
\norm{\V v(t)-\V v^*(t)}_{L^2(\Ell)}+|\V \omega_1(t)-\V \omega_1^*(t)|+|\V \omega(t)-\V \omega^*(t)|
\\
\le c\left( \norm{\V v_0-\V v^*_0}_{L^2(\Ell)}+|\V \omega_{10}-\V \omega_{10}^*|+|\V \omega_0-\V \omega^*_0|\right),
\quad \text{ for all }t\in[0,T]. 
\end{multline*}
\item[b)] If $(\V v_0,\V \omega_{10},\V \omega_0)=(\V v_0^*,\V \omega_{10}^*,\V \omega_0^*)$, then 
$(\V v,\V \omega_1,\V \omega)=(\V v^*,\V \omega_1^*,\V \omega^*)$ a.e. in $[0,T]\times \Ell$. 
\end{itemize}
\end{theorem}

To show the previous theorem, we need some preliminary lemmas. Their proofs are standard, they are similar to the ones provided in \cite[Chapter 3]{Ma}. 

\begin{lemma}\label{lem:equivalent_weak}
Consider a weak solution $(\V v,\V \omega_1,\V \omega)$ of \eqref{eq:Motion} and the extension $\V{\tilde v}$ of $\V v$ defined in \eqref{eq:extension}. Then,  $\V{\tilde v}$  can be redefined on a set of zero Lebesgue measure in such a way that $\V{\tilde v}\in L^2_R(\mathcal V)$ for all $t\in [0,T)$ and it satisfies the following equation   
\begin{equation}\label{eq:weak_v_d}\begin{split}
&-\int^t_s\left[(\V{\tilde v},\frac{\partial \V \phi}{\partial t})_B-b\left(\frac{\partial \V \phi}{\partial t}\right)\cdot \T I \cdot \V \Omega\right]\d \tau
\\
&\qquad+(\V{\tilde v}(t), \V \phi(t))_B-b(\V \phi(t))\cdot\T I \cdot\V \Omega(t)-(\V{\tilde v}(s), \V \phi(s))_B+b(\V \phi(s))\cdot\T I \cdot\V \Omega(s)
\\
&\qquad+2\mu\int^t_s\int_{\mathcal V}\T D(\V{\tilde v})\dotdot \T D(\V \phi)\;\d V\d \tau
+ \int^t_s\int_{\mathcal V}\tilde \rho[\V{\tilde v}\cdot \nabla \V{\tilde v}+2(\V \Omega+b(\V{\tilde v}))\times \V{\tilde v}]\cdot \V \phi\;\d V\d \tau
=0, 
\end{split}\end{equation}
for all $0\le s\le t$, $t<T$ and $\V \phi\in\mathcal D_R(\mathcal V_T)$. 
\end{lemma}

For a Banach space $X$, we will consider the {\em (time-)mollification} $\V w_h$ of a function $\V w\in L^2(0,T;\mathcal H_2^1(\mathcal V))$ as the function defined by 
\[
\V w_h(\V x,t):=\int^T_0 j_h(t-s)\V w(\V x,s)\; \d s\in C^\infty([0,T];\mathcal H_2^1(\mathcal V)),
\]
where $\{j_h\in C^\infty_0(-h,h):\, 0<h<T\}$ is a family of mollifiers. Then, the following lemma is an immediate consequence of \cite[Theorem 2.29]{Adams} and \cite[Lemma 1.3.3. \& Remark 1.3.8 (b)]{ArBaHiNe}. 

\begin{lemma}\label{lem:mollification}
Let $H$ be a Hilbert space with the inner product $\langle\cdot,\cdot\rangle$. If $\V u\in C_w([0,T),H)$, then 
\[
\lim_{h\to 0}\langle\V u-\V u_h,\V \psi\rangle=0
\] 
uniformly on every closed interval $J\subset [0,T)$ and for every $\V \psi\in H$. 

Let $X$ be a Banach space. For every $\V w\in L^p(0,T;X)$, $1\le p<\infty$, 
\[
\lim_{h\to 0}\norm{\V w-\V w_h}_{L^p(0,T;X)}=0.
\]
Moreover, let $\{\V w_n\}_{n\in \N}$be a sequence converging to $\V w$ in $L^p(0,T;X)$. Then, 
\[
\lim_{n\to \infty}\norm{(\V v_n)_h-\V w_h}_{L^p(0,T;X)}=0,\qquad \text{ for all }0<h<T. 
\]
\end{lemma}

Moreover, the following result holds. 
\begin{lemma}
For every $\V u, \V w\in C_w([0,T);L^2_R(\mathcal V))\cap L^2(0,T;L^2_R(\mathcal V))$
\begin{equation}\label{eq:d_mollification}
\lim_{h\to 0}\int^t_0\left(( \V u,\frac{\partial \V w_h}{\partial \tau})_B+(\frac{\partial \V u_h}{\partial \tau},\V w)_B\right)\; \d \tau=( \V u(t),\V w(t))_B-(\V u(0),\V w(0))_B 
\end{equation}
$t\in [0,T)$.
\end{lemma}

\begin{lemma}\label{lem:approximation}
$\mathcal D_R(\mathcal V_T)$ is dense in $L^2(0,T;\mathcal H^1_2(\mathcal V))$. In particular, every $\V w\in L^2(0,T;\mathcal H^1_2(\mathcal V))$ can be approximated in $L^2(0,T;\mathcal H^1_2(\mathcal V))$ by the family $\{\V w_{n,h}:\; n\in \N,\, 0<h<T\}$ of functions 
\[
\V w_{n,h}:=\sum^n_{k=1}(\V w_h,\V \Psi_k)_1\V \Psi_k,
\]
where $\{\V \Psi_k\}_{k\in \N}\subset \mathcal D_R(\mathcal V)$ is a basis of $\mathcal H_1(\mathcal V)$. Moreover, the following convergences hold: 
\begin{equation*}
\begin{aligned}
&\lim_{n\to \infty}\norm{\V w_{n,h}-\V w_h}_{1,2}=0\qquad &&\text{ for all }t\in [0,T]\,\text{ and }\,h<T,
\\
&\lim_{n\to \infty}\norm{\V w_{n,h}-\V w_h}_{L^2(0,T;\mathcal H^1_2(\mathcal V))}=0\qquad &&\text{ for all }h<T, 
\\
&\lim_{h\to 0}\left(\lim_{n\to \infty}\norm{\V w_{n,h}-\V w}_{L^2(0,T;\mathcal H^1_2(\mathcal V))}\right)=0. 
\end{aligned}
\end{equation*}
\end{lemma} 

We are now in position to prove Theorem \ref{th:continuous_dependence}
\begin{proof}[Proof of Theorem \ref{th:continuous_dependence}]
Consider the extensions $\V{\tilde v}$ and $\V{\tilde v}^*$ of $\V v$ and $\V v^*$ (together with the corresponding initial conditions), defined in \eqref{eq:extension}, respectively. Set $\V \Omega=\V \omega_1-b(\V{\tilde v})$ and $\V \Omega^*=\V \omega_1^*-~b(\V{\tilde v}^*)$. 
Let $\{\V{\tilde v}_{n,h}:\,n\in \N,\, 0<h<T\}$ and $\{\V{\tilde v}^*_{n,h}:\,n\in \N,\, 0<h<T\}$ be the approximating families of $\V{\tilde v}$ and $\V{\tilde v}^*$ in $L^2(0,T;\mathcal H^1_2(\mathcal V))$ given by Lemma \ref{lem:approximation}, respectively. For every $n\in \N$ and $h\in (0,T)$, let us replace $\V{\tilde v}^*_{n,h}$ and $\V{\tilde v}_{n,h}$  in place of $\V \phi$ in \eqref{eq:weak_v_d} with $s=0$, for $\V{\tilde v}$ and $\V{\tilde v}^*$, respectively. The following equations hold: 
\begin{equation*}\begin{split}
&-\int^t_0\left[(\V{\tilde v},\frac{\partial \V{\tilde v}^*_{n,h}}{\partial \tau})_B-b\left(\frac{\partial \V{\tilde v}^*_{n,h}}{\partial \tau}\right)\cdot\T I\cdot\V \Omega\right]\; \d \tau
+(\V{\tilde v}(t),\V{\tilde v}^*_{n,h}(t))_B
-(\V{\tilde v}_0,\V{\tilde v}^*_{n,h}(0))_B
\\
&\quad -b(\V{\tilde v}^*_{n,h}(t))\cdot\T I \cdot\V \Omega(t)+b(\V{\tilde v}^*_{n,h}(0))\cdot\T I \cdot\V \Omega_0
+2\mu\int^t_0\int_{\mathcal V}\T D(\V{\tilde v})\dotdot \T D(\V{\tilde v}^*_{n,h})\;\d V\d \tau
\\
&\quad+\int^t_0\int_{\mathcal V}\tilde \rho[\V{\tilde v}\cdot \nabla \V{\tilde v}
+2(\V \Omega+b(\V{\tilde v}))\times \V{\tilde v}]\cdot \V{\tilde v}^*_{n,h}\;\d V\d \tau
=0,
\end{split}\end{equation*}
and 
\begin{equation*}\begin{split}
&-\int^t_0\left[(\V{\tilde v}^*,\frac{\partial \V{\tilde v}_{n,h}}{\partial \tau})_B-b\left(\frac{\partial \V{\tilde v}_{n,h}}{\partial \tau}\right)\cdot\T I\cdot\V \Omega^*\right]\; \d \tau
+(\V{\tilde v}^*(t), \V{\tilde v}_{n,h}(t))_B
-(\V{\tilde v}^*_0, \V{\tilde v}_{n,h}(0))_B
\\
&\quad -b(\V{\tilde v}_{n,h}(t))\cdot\T I \cdot\V \Omega^*(t)+b(\V{\tilde v}_{n,h}(0))\cdot\T I \cdot\V \Omega^*_0
+2\mu\int^t_0\int_{\mathcal V}\T D(\V{\tilde v}^*)\dotdot \T D(\V{\tilde v}_{n,h})\;\d V\d \tau
\\
&\quad+ \int^t_0\int_{\mathcal V}\tilde \rho[\V{\tilde v}^*\cdot \nabla \V{\tilde v}^*
+2(\V \Omega^*+b(\V{\tilde v}^*))\times \V{\tilde v}^*]\cdot \V{\tilde v}_{n,h}\;\d V\d \tau
=0. 
\end{split}\end{equation*}
Taking the limit as $n\to \infty$ in the preceding two equations, we find that 
\begin{equation}\label{eq:approx_1_h}
\begin{split}
&-\int^t_0\left[(\V{\tilde v},\frac{\partial \V{\tilde v}^*_{h}}{\partial \tau})_B-b\left(\frac{\partial \V{\tilde v}^*_{h}}{\partial \tau}\right)\cdot\T I\cdot\V \Omega\right]\; \d \tau
+(\V{\tilde v}(t),\V{\tilde v}^*_{h}(t))_B
-(\V{\tilde v}_0,\V{\tilde v}^*_{h}(0))_B
\\
&\quad -b(\V{\tilde v}^*_{h}(t))\cdot\T I \cdot\V \Omega(t)+b(\V{\tilde v}^*_{h}(0))\cdot\T I \cdot\V \Omega_0
+2\mu\int^t_0\int_{\mathcal V}\T D(\V{\tilde v})\dotdot \T D(\V{\tilde v}^*_{h})\;\d V\d \tau
\\
&\quad+\int^t_0\int_{\mathcal V}\tilde \rho[\V{\tilde v}\cdot \nabla \V{\tilde v}
+2(\V \Omega+b(\V{\tilde v}))\times \V{\tilde v}]\cdot \V{\tilde v}^*_{h}\;\d V\d \tau
=0,
\end{split}
\end{equation}
and 
\begin{equation}\label{eq:approx_2_h}
\begin{split}
&-\int^t_0\left[(\V{\tilde v}^*,\frac{\partial \V{\tilde v}_{h}}{\partial \tau})_B-b\left(\frac{\partial \V{\tilde v}_{h}}{\partial \tau}\right)\cdot\T I\cdot\V \Omega^*\right]\; \d \tau
+(\V{\tilde v}^*(t), \V{\tilde v}_{h}(t))_B-(\V{\tilde v}^*_0, \V{\tilde v}_{h}(0))_B
\\
&\quad -b(\V{\tilde v}_{h}(t))\cdot\T I \cdot\V \Omega^*(t)+b(\V{\tilde v}_{h}(0))\cdot\T I \cdot\V \Omega^*_0
+2\mu\int^t_0\int_{\mathcal V}\T D(\V{\tilde v}^*)\dotdot \T D(\V{\tilde v}_{h})\;\d V\d \tau
\\
&\quad+ \int^t_0\int_{\mathcal V}\tilde \rho[\V{\tilde v}^*\cdot \nabla \V{\tilde v}^*
+2(\V \Omega^*+b(\V{\tilde v}^*))\times \V{\tilde v}^*]\cdot \V{\tilde v}_{h}\;\d V\d \tau
=0. 
\end{split}
\end{equation}
In the previous limits, the convergence of the linear terms is standard thanks to Lemma \ref{lem:approximation}. For what concerns the nonlinear terms, the convergence follows from the following estimates, Lemma \ref{lem:approximation} and Lebesgue dominated convergence theorem. For every $\V u_1,\V u_2\in L^\infty(0,T;\mathcal H(\mathcal V))\cap L^2(0,T;\mathcal H_2^1(\mathcal V))$: 
\begin{align*}
\int^t_0\int_{\mathcal V}\tilde \rho (\V u_1\cdot \nabla \V u_1)\cdot [(\V u_2)_{n,h}-(\V u_2)_h]\;\d V\d \tau
&\le 
\int^t_0\norm{\V u_1}_{6}\norm{\nabla \V u_1}_{2}\norm{(\V u_2)_{n,h}-(\V u_2)_h}_{3}\;\d \tau
\\
&\le c_1\int^t_0\norm{\nabla \V{\tilde u_1}}^2_{2}\norm{(\V u_2)_{n,h}-(\V u_2)_h}_{1,2}\;\d \tau
\end{align*}
by H\"older inequality, \eqref{eq:sobolev_korn}, \eqref{eq:korn_q} and Sobolev embedding theorem. Moreover, for every $\V a\in L^\infty(0,T)$, by H\"older inequality and \eqref{eq:korn_q}
\begin{equation}\label{eq:nonlinear_estimate0}
\begin{split}
\int^t_0\int_{\mathcal V}2\tilde \rho(\V a+b(\V u_1)\times \V u_1)\cdot [(\V u_2)_{n,h}-(\V u_2)_h]\;\d V\d \tau
&\le 
\int^t_0\norm{\V a\times \V u_1}_2\norm{(\V u_2)_{n,h}-(\V u_2)_h}_{1,2}\; \d \tau
\\
&\le c_2\int^t_0\norm{(\V u_2)_{n,h}-(\V u_2)_h}_{1,2}^2,\;\d \tau,
\end{split}
\end{equation}
where $c_2$ is a positive constant depending on $\norm{\V u_1}_{L^2(0,T;\mathcal H^1_2(\mathcal V))}$ and 
$\max_{t\in [0,T]}|\V a(t)|$. 

From \eqref{eq:weak}$_2$ for $\V \Omega$ and $\V \Omega^*$, we find that 
\begin{multline*}
\int^t_0 b\left(\frac{\partial \V{\tilde v}^*_{h}}{\partial \tau}\right)\cdot\T I\cdot\V \Omega\; \d \tau
-b(\V{\tilde v}^*_{h}(t))\cdot\T I \cdot\V \Omega(t)+b(\V{\tilde v}^*_{h}(0))\cdot\T I \cdot\V \Omega_0
\\
=\int^t_0 b(\V{\tilde v}^*_{h})\cdot [(\V \Omega+b(\V{\tilde v}))\times \T I\cdot \V \Omega]\; \d \tau
\end{multline*}
and
\begin{multline*}
\int^t_0 b\left(\frac{\partial \V{\tilde v}_{h}}{\partial \tau}\right)\cdot\T I\cdot\V \Omega^*\; \d \tau
-b(\V{\tilde v}_{h}(t))\cdot\T I \cdot\V \Omega^*(t)+b(\V{\tilde v}_{h}(0))\cdot\T I \cdot\V \Omega^*_0
\\ =\int^t_0 b(\V{\tilde v}_{h})\cdot [(\V \Omega^*+b(\V{\tilde v}^*))\times \T I\cdot \V \Omega^*]\; \d \tau.
\end{multline*}

Hence, adding \eqref{eq:approx_1_h} and \eqref{eq:approx_2_h}, we find that 
\begin{equation}\label{eq:approx_3_h}
\begin{split}
&-\int^t_0\left[(\V{\tilde v},\frac{\partial \V{\tilde v}^*_{h}}{\partial \tau})_B+(\V{\tilde v}^*,\frac{\partial \V{\tilde v}_{h}}{\partial \tau})_B\right]\;\d \tau +(\V{\tilde v}(t),\V{\tilde v}^*_{h}(t))_B-(\V{\tilde v}_0,\V{\tilde v}^*_{h}(0))_B
\\
&\qquad +(\V{\tilde v}^*(t), \V{\tilde v}_{h}(t))_B-(\V{\tilde v}^*_0, \V{\tilde v}_{h}(0))_B
\\
&\qquad+\int^t_0 b(\V{\tilde v}^*_{h})\cdot [(\V \Omega+b(\V{\tilde v}))\times \T I\cdot \V \Omega]\; \d \tau
+\int^t_0 b(\V{\tilde v}_{h})\cdot [(\V \Omega^*+b(\V{\tilde v}^*))\times \T I\cdot \V \Omega^*]\; \d \tau
\\
&\quad\quad+2\mu\int^t_0\int_{\mathcal V}[\T D(\V{\tilde v})\dotdot \T D(\V{\tilde v}^*_{h})
+\T D(\V{\tilde v}^*)\dotdot \T D(\V{\tilde v}_{h})]\;\d V\d \tau
\\
&\quad\quad+\int^t_0\int_{\mathcal V}\tilde \rho[\V{\tilde v}\cdot \nabla \V{\tilde v}
+2(\V \Omega+b(\V{\tilde v}))\times \V{\tilde v}]\cdot \V{\tilde v}^*_{h}\;\d V\d \tau
\\
&\quad\quad+ \int^t_0\int_{\mathcal V}\tilde \rho[\V{\tilde v}^*\cdot \nabla \V{\tilde v}^*
+2(\V \Omega^*+b(\V{\tilde v}^*))\times \V{\tilde v}^*]\cdot \V{\tilde v}_{h}\;\d V\d \tau
=0. 
\end{split}
\end{equation}

Next, we take the limit as $h\to 0$ in \eqref{eq:approx_3_h}. Again, the convergence of the linear terms follows easily thanks to \eqref{eq:d_mollification} and Lemma \ref{lem:mollification}. For what concerns the nonlinear terms, we use \eqref{eq:nonlinear_estimate0} and the following classical inequality 
\begin{multline}\label{eq:nonlinear_estimate_rs}
\left|\int^T_0\int_{\mathcal V}\tilde \rho (\V u_1\cdot \nabla \V u_2)\cdot \V u_3\;\d V\d \tau\right|
\\
\le c \left(\int^T_0\norm{\nabla \V u_1}_2^2\;\d \tau\right)^{3/2q}\left(\int^T_0\norm{\nabla \V u_2}_2^2\;\d \tau\right)^{1/2}\left(\int^T_0\norm{\V u_3}^p_q\norm{\V u_1}^2_2\;\d \tau\right)^{1/p}
\end{multline}
which holds for every $\V u_1,\V u_2\in L^\infty(0,T;\mathcal H(\mathcal V))\cap L^2(0,T;\mathcal H^1_2(\mathcal V))$ and $\V u_3\in L^p(0,T;L^q(\mathcal V))$ with $p$ and $q$ satisfying \eqref{eq:serrin} (see \cite[Lemma 1]{Serrin}).  Moreover, from \eqref{eq:weak}$_2$, we find that 
\begin{multline*}
\V \Omega^*(t)\cdot \T I \cdot \V \Omega(t)- \V \Omega^*_0\cdot \T I \cdot \V \Omega_0
\\
=-\int^t_0\V \Omega^*\cdot [(\V \Omega+b(\V{\tilde v}))\times \T I\cdot \V \Omega]\; \d \tau
-\int^t_0\V \Omega\cdot [(\V \Omega^*+b(\V{\tilde v}^*))\times \T I\cdot \V \Omega^*]\; \d \tau.
\end{multline*}
Hence, the couples $(\V{\tilde v},\V \Omega)$ and $(\V{\tilde v}^*,\V \Omega^*)$ satisfy the following equality 
 \begin{equation}\label{eq:2vv}
\begin{split}
(\V{\tilde v}(t),\V{\tilde v}^*(t))_B&-(\V{\tilde v}_0,\V{\tilde v}^*_0)_B
+\int^t_0 [\V \Omega^*+b(\V{\tilde v}^*_{h})]\cdot [(\V \Omega+b(\V{\tilde v}))\times \T I\cdot (\V \Omega-\V \Omega^*)]\; \d \tau
\\
&+\V \Omega^*(t)\cdot \T I \cdot \V \Omega(t)- \V \Omega^*_0\cdot \T I \cdot \V \Omega_0+4\mu\int^t_0\int_{\mathcal V}\T D(\V{\tilde v})\dotdot \T D(\V{\tilde v}^*)\;\d V\d \tau
\\
&+\int^t_0\int_{\mathcal V}\tilde \rho[(\V{\tilde v}-\V{\tilde v}^*)\cdot \nabla \V{\tilde v}
+2(\V \Omega-\V \Omega^*+b(\V{\tilde v}-\V{\tilde v}^*))\times \V{\tilde v}]\cdot \V{\tilde v}^*\;\d V\d \tau
=0. 
\end{split}
\end{equation}

We recall that, by Definition \ref{def:weak}, $(\V{\tilde v},\V \Omega)$ and $(\V{\tilde v}^*,\V \Omega^*)$ satisfy the strong energy inequality \eqref{eq:strong_energy} for all $t\in[0,T]$: 
\begin{equation}\label{eq:strong_energy_v}
\mathcal E(\V{\tilde v}(t))+\V \Omega(t)\cdot \T I\cdot \V \Omega(t)+4\mu\int^t_s\norm{\T D( \V{\tilde v}(\tau))}^2_{L^2(\Ell)}\;\d \tau\le\mathcal E(\V{\tilde v}_0)+\V \Omega_0\cdot \T I\cdot \V \Omega_0,
\end{equation}
and 
\begin{equation}\label{eq:strong_energy_vstar}
\mathcal E(\V{\tilde v}^*(t))+\V \Omega^*(t)\cdot \T I\cdot \V \Omega^*(t)+4\mu\int^t_s\norm{\T D( \V{\tilde v}^*(\tau))}^2_{L^2(\Ell)}\;\d \tau\le\mathcal E(\V{\tilde v}_0)^*+\V \Omega^*_0\cdot \T I\cdot \V \Omega^*_0.
\end{equation}
Adding \eqref{eq:strong_energy_v} and \eqref{eq:strong_energy_vstar}, and subtracting twice of \eqref{eq:2vv}, we find that 
the fields $\V w:=\V{\tilde v}-\V {\tilde v}^*$ and $\V \xi:=\V \Omega-\V\Omega^*$ must satisfy the following inequality
\begin{equation}\label{eq:subtraction}
\begin{split}
\mathcal E(\V w(t))+\V \xi(t)\cdot \T I\cdot \V \xi(t)&+4\mu\int^t_s\norm{\T D( \V w(\tau))}^2_{L^2(\Ell)}\;\d \tau
\\
&\le
\mathcal E(\V w_0)+\V \xi_0\cdot \T I\cdot \V \xi_0-2\int^t_0 [\V \xi+b(\V w)]\cdot [(\V \Omega^*+b(\V{\tilde v}^*_{h}))\times \T I\cdot \V \xi]\; \d \tau
\\
&\quad+2\int^t_0\int_{\mathcal V}\tilde \rho[\V w\cdot \nabla \V w+2(\V \xi+b(\V w))\times \V w]\cdot \V{\tilde v}^*\;\d V\d \tau,
\end{split}
\end{equation}
where $\V w_0:=\V{\tilde v}_0-\V {\tilde v}^*_0$ and $\V \xi_0:=\V \Omega_0-\V\Omega^*_0$. 
By H\"older inequality, \eqref{eq:nonlinear_estimate_rs} and Young's inequality, we get the following estimates 
\begin{multline*}
\mathcal E(\V w(t))+\V \xi(t)\cdot \T I\cdot \V \xi(t)+2\mu\int^t_s\norm{\T D( \V w(\tau))}^2_{L^2(\Ell)}\;\d \tau
\le \mathcal E(\V w_0)+\V \xi_0\cdot \T I\cdot \V \xi_0
\\+ c_3\int^t_0[\norm{\V{\tilde v}^*(\tau)}^p_{L_q(\mathcal V)}+\norm{\V w(\tau)}_{L^2(\mathcal V)}+|\V \xi(\tau)|][\mathcal E(\V w(\tau))+\V \xi(\tau)\cdot \T I\cdot \V \xi(\tau)] \;\d \tau.
\end{multline*}
Recalling \eqref{eq:extension} and \eqref{eq:energy_d} and using Gr\"onwall's Lemma together with \eqref{eq:kokr}, properties {\em (a)} and {\em (b)} of Theorem \ref{th:continuous_dependence} immediately follow. 
\end{proof}

\section{Existence of strong solution}\label{sec:strong}

In this section, we will prove the local in time existence and continuous dependence upon initial data of strong solutions to \eqref{eq:Motion} for a considerably ``large'' class of initial conditions. The approach is the one of maximal $L^p-L^q$ regularity in time-weighted $L^p$-spaces (see Appendix \ref{sec:maximal_regularity} for a brief discussion on such approach).  

Let us introduce some notation. For the remaining part of the paper, the brackets $[\cdot,\cdot]_\theta$ denote the complex interpolation, whereas $(\cdot,\cdot)_{\alpha,\gamma}$ are used for the real interpolation. For $p\in (1,\infty)$, $1/p<\upmu\le 1$ and a Banach space $X$, the {\em time-weighted $L^p$-spaces} are defined as follows 
\begin{equation}\label{eq:time_weigthed_Lp}
\begin{aligned}
&\V u\in L^p_{\upmu}((0,T); X) && \Leftrightarrow \quad  t^{1-\upmu}\V u\in L^p((0,T);X),\\
&\V u\in H^1_{p,\upmu}((0,T);X) && \Leftrightarrow \quad \V u,\d \V u/\d t\in L^p_{\upmu}((0,T); X).
 \end{aligned}
\end{equation}
Consider the operator $(\T A_q,D(\T A_q))$ where 
\begin{equation}\label{eq:stokes_v}
\T A_q:=-\frac{\mu}{\tilde \rho}\;\mathcal P_q\Delta
\end{equation}
is the Stokes operator with domain $D(\T A_q):=\{\V{\tilde w}\in \mathcal H^2_q(\mathcal V)\cap H_q(\mathcal V):\; \V w=\V 0\text{ on }\mathcal C\}$, $\tilde \rho$ is given in  \eqref{eq:density}; we recall that $\mathcal P_q$ is the projection of $L^q_R(\mathcal V)$ onto $\mathcal H_q(\mathcal V)$. Moreover, for $p,\;q\in(1,\infty)$, we consider the spaces 
$X_0:=\mathcal H_q(\mathcal V)\times \R^3$, $X_1:=D(\T A_q)\times\R^3$, and  the  interpolation spaces 
\[
X_{\gamma,\upmu}:=(X_0,X_1)_{\upmu-1/p,p},\quad X_\alpha=[X_0,X_1]_{\alpha}\quad\text{for }\mu\in (1/p,1],\;\alpha\in (0,1). 
\]  
The previous spaces are endowed with the norms 
\[
\norm{\V u}_{X_0}:=\sqrt{\norm{\V{\tilde v}}_{L^q(\mathcal V)}^2+|\V \omega_1|^2},\qquad\qquad 
\norm{\V u}_{X_1}:=\sqrt{\norm{\V{\tilde v}}_{W^{2,q}(\mathcal V)}^2+|\V \omega_1|^2}
\]
and similarly for the interpolation spaces. We recall the following characterization of Besov spaces $B^s_{qp}(\mathcal V)=(H^{s_0}_q(\mathcal V),H^{s_1}_q(\mathcal V))_{\theta,p}$ as real interpolation of Bessel potential spaces, and of Bessel potential spaces $H^s_q(\mathcal V)=[H^{s_0}_q(\mathcal V),H^{s_1}_q(\mathcal V)]_{\theta}$. These characterizations are valid for $s_0\ne s_1\in \R$, $p,q\in [1, \infty)$, $\theta\in (0,1)$ and $s=(1-\theta)s_0+\theta s_1$. We also recall that $B^s_{qq}(\mathcal V) = W^{s,q}(\mathcal V)$  and $B^s_{22}(\mathcal V) = W^{s,2}(\mathcal V) =H^s_2(\mathcal V)$. 

Before stating our main result about existence and related properties of strong solutions to \eqref{eq:Motion}, we need some preliminary observations. Let us consider the initial boundary value problem which describes the motion of a rigid body having a cavity $\mathcal V$ completely filled by a viscous liquid with a varying density $\tilde \rho$ defined in  \eqref{eq:density}.
\begin{equation}\label{eq:Motion_e}
\begin{aligned}
&\left.\begin{split}
&\frac{\partial \V{\tilde v}}{\partial t}+\V{\dot \omega}_1\times \V x+
\V{\tilde v}\cdot \nabla \V{\tilde v}+2\V \omega_1\times \V{\tilde v}
=\frac{\mu}{\tilde \rho}\Delta \V{\tilde v}-\frac{1}{\tilde \rho}\nabla \pi
\\ 
&\diver \V{\tilde v}=0
\end{split}\right\}&&\text{ on }\mathcal V\times (0,\infty),
\\
&\ \V{\dot \omega}_1-b\left(\frac{\partial \V{\tilde v}}{\partial t}\right)+\T I^{-1}\cdot\left[\V \omega_1\times \T I\cdot (\V \omega_1-b(\V{\tilde v}))\right]=\V 0 
&&\text{ in }(0,\infty),
\\
&\ \V{\tilde v}=\V 0&&\text{ on }\mathcal C,
\\
&\ \V{\tilde v}|_{t=0}=\V{\tilde v}_0,\qquad \V \omega_1(0)=\V \omega_{10}&& 
\end{aligned}
\end{equation}
Assume that for some initial data $(\V{\tilde v}_0, \V \omega_{10})$ satisfying the condition 
\begin{equation}\label{eq:compatibility}
\V{\tilde v}_0=\left\{\begin{aligned}
&\V v_0&&\text{ on }\Ell,
\\
&\V \omega_0\times \V x&&\text{ on }\mathcal B_2,
\end{aligned}\right.\qquad \qquad \text{with }\V v_0=\V \omega_0\times \V x\text{ on }\mathcal S,
\end{equation}
there exists $(\V{\tilde v},\V \omega_1)$ a strong solution to \eqref{eq:Motion_e} in the class $\mathbb{E}_{1,\upmu}(0,T)$ with $\upmu=1$, defined in \eqref{eq:regularity_c} below. Then there exist $\V v\in H^1_p(0,t_1;H_q(\Ell))\cap L^p(0,t_1;H^2_q(\Ell))$ and $\V \omega\in C^1((0,T];\R^3)$ such that $\V v=\V 0$ on $\mathcal C$, $\V v=\V \omega\times \V x$ on $\mathcal S$, and 
\[
\V{\tilde v}= \left\{\begin{aligned}
&\V v &&\text{on }\Ell,
\\
&\V \omega\times \V x &&\text{on }\mathcal B_2.
\end{aligned}\right.
\]
Using a duality argument (generalizing that in Remark \ref{re:strong_2}), one can find that the triple $(\V v,\V \omega_1,\V \omega)$ is a strong solution to \eqref{eq:Motion}. Moreover, $(\V v,\V \omega_1,\V \omega)$ satisfies the initial conditions thanks to \eqref{eq:compatibility}. Therefore, the goal of this section is to investigate the existence and related properties of strong solutions to \eqref{eq:Motion_e}. 

In the following we set 
\begin{equation*}
\mathcal B^{s}_{qp,\sigma}(\mathcal V):=
\left\{ 
\begin{aligned}
&\{\V u\in  B^{s}_{qp}(\mathcal V)\cap \mathcal H_q(\mathcal V): \V u=\V 0\; \text{on}\; \mathcal C\}, &&s>1/q,\\
&B^{s}_{qp}(\mathcal V)\cap \mathcal H_q(\mathcal V), && s\in [0,1/q).\\
\end{aligned}
\right.
\end{equation*}

In view of the previous observations, next theorem turns out to be the main result of this section. 

\begin{theorem}\label{th:strong}
Suppose
\begin{equation}
\label{assumptions-pq}
p\in(1,\infty),\quad q\in (1,3),\quad  2/p +3/q\le 3,
\end{equation}
and let (the time-weight) $\upmu$ satisfy
\begin{equation}
\label{assumptions-mu}
\upmu\in (1/p,1],\quad  \upmu\ge \upmu_{\rm crit}=\frac{1}{p} + \frac{3}{2q}-\frac{1}{2}. 
\end{equation}
\begin{enumerate}
\setlength\itemsep{1mm}
\item[{\bf (a)}]
Let  $\V u_0=(\V{\tilde v}_0,\V \omega_{10})\in \mathcal B^{2\upmu-2/p}_{qp,\sigma}(\mathcal V)\times \R^3=X_{\gamma,\upmu}$ be given such that \eqref{eq:compatibility} is satisfied.
Then there are positive constants $T=T(\V u_0)$ and $\eta=\eta(\V u_0)$ such that
\eqref{eq:Motion_e} admits a unique solution $\V u(\cdot, \V u_0)=(\V{\tilde v},\V \omega_1)$ in 
 \begin{equation}\label{eq:regularity_c}
 \mathbb E_{1,\upmu}(0,T)=H^1_{p,\upmu}((0,T); X_0)
 \cap L^p_{\upmu}((0,T); X_1).
 \end{equation}
\item[{\bf(b)}]
Suppose $p_j, q_j$, $\upmu_j$ satisfy \eqref{assumptions-pq}-\eqref{assumptions-mu}  and, in addition, $p_1\leq p_2$, $q_1\leq q_2$ as well as
\begin{equation}\label{mu-j}
 \upmu_1- \frac{1}{p_1}- \frac{3}{2q_1} \ge  \upmu_2- \frac{1}{p_2}- \frac{3}{2q_2}.
 \end{equation}
Then for each initial value $(\V{\tilde v}_0,\V \omega_{10})\in \mathcal B^{2\upmu_1 -2/p_1}_{q_1 p_1,\sigma}(\mathcal V)\times \R^3$ satisfying \eqref{eq:compatibility}, problem  \eqref{eq:Motion_e} admits a unique solution $(\V{\tilde v},\V \omega_1)$ in the class
\begin{equation*}
\begin{split}
&H^1_{p_1,\upmu_1}((0,T); \mathcal H_{q_1}(\mathcal V)\times\R^3)\cap L_{\upmu_1}^{p_1}((0,T); D(\T A_{q_1})\times\R^3) \\
&\cap H^1_{p_2,\upmu_2}((0,T); \mathcal H_{q_2}(\mathcal V)\times\R^3)\cap L_{\upmu_2}^{p_2}((0,T);D(\T A_{q_2})\times\R^3).
\end{split}
\end{equation*}
\item[{\bf (c)}]
Each solution exists on a maximal interval  $[0,t_+)=[0,t_+(\V u_0))$, and enjoys the additional regularity property
\begin{equation*}
\V{\tilde v} \in C([0,t_+); \mathcal B^{2\upmu-2/p}_{qp,\sigma}(\mathcal V))\cap C((0,t_+);\mathcal B^{2-2/p}_{qp,\sigma}(\mathcal V)),
\quad \V \omega_1\in C^1([0,t_+),\R^3).
\end{equation*}
\item[{\bf (d)}]
The solution $\V u=(\V{\tilde v},\V \omega_1)$ exists globally if
 $\V u([0,t_+))\subset B^{2\upmu-2/p}_{qp}(\mathcal V)\times \R^3$ is relatively compact. 
\end{enumerate}
\end{theorem}

\begin{proof}
The statements in (a), (c) and (d) follow from Theorem \ref{th:MRSP}. We will verify the hypotheses of Theorem \ref{th:MRSP} in the next three steps. 

\paragraph{Step 1. A semilinear evolution equation}Problem \eqref{eq:Motion_e} can be reformulated as a semilinear evolution equation for the variable $\V u=[\V{\tilde v},\V \omega_1]^T$: 
\begin{equation}\label{eq:evolution0}
\T E\cdot \frac{\d \V u}{\d t}+\T A \V u=\T G(\V u,\V u),\quad\V u(0)=\V u_0,
\end{equation}
where, 
\begin{equation}\label{eq:E}
\begin{split}
&\T E:\;\left[\begin{matrix}\V w \\ \V \xi\end{matrix}\right]\in X_0\mapsto\T E(\V w,\V \xi):=\left[\begin{matrix}
\V w+\mathcal P_q\left(\V \xi\times \V x\right)
\\
\V\xi-b(\V w)\end{matrix}\right]\in X_0,
\\
&\T A:=\left[\begin{matrix}
\T A_q & \T 0
\\
\T 0 & \T 0
\end{matrix}\right]:\; X_1\to X_0,\quad\text{$\T A_q$ defined in \eqref{eq:stokes_v}, }
\\
&\T G(\V u,\V u):=\left[\begin{matrix}
\mathcal P_q(-\V{\tilde v}\cdot \nabla \V{\tilde v}-2\V \omega_1\times \V{\tilde v})
\\
-\T I^{-1}\cdot[\V \omega_1\times \T I\cdot(\V \omega_1-b(\V{\tilde v}))
\end{matrix}\right],
\end{split}\end{equation}
and the functional $b(\cdot)$ has been introduced in \eqref{eq:a_varphi}. 
The operator $\T E$ is linear, bounded, invertible, and has a bounded inverse. The linearity and boundedness of $\T E$ is obvious from its definition. For what concerns its invertibility, we observe that $\T E=\T 1+\T K$ with  
\[
\T K:=\left[\begin{matrix}
\T 0 & \mathcal P_q(\cdot \times \V x)
\\
-b(\cdot) & \T 0
\end{matrix}\right]
\]
a bounded operator with a finite dimensional range (see \eqref{eq:a_varphi}). A basis for the range of $\T K$ is given by $\{(\V e_i,\mathcal P_q(\V e_i\times \V x)):\; i=1,2,3\}$). Thus, $\T K$ is a compact operator, and $\T E$ is a Fredholm operator of index zero (by \cite[Theorem 5.26, page 238]{Kato}). The invertibility of $\T E$ then follows if we prove that its null space reduces to $\mathsf{N}[\T E]=\{\V 0\}$. The latter immediately follows from Lemma \ref{le:kokr} (actually, from its proof). In fact, $\T E$ is one-to-one when $q=2$. In addition to the previous properties of $\T E$, we can also infer that $\T E^{-1}\equiv\T 1+\T C$, where $\T C:=-\T K\cdot\T E^{-1}:\; X_0\to \mathcal R(\mathcal V)\cap \mathcal H_q(\mathcal V)\times\R^3$ is a bounded operator with a finite dimensional range, and then compact. 

\noindent Let us consider the linear operator $\L:=\T E^{-1}\cdot \T A$ with domain $X_1$, and observe that 
\begin{equation}\label{eq:op_L}
\L=(\T 1+\T C)\cdot \T A=\left[\begin{matrix}
\T A_q & 0
\\
0 & 0
\end{matrix}\right]+\T C\left[\begin{matrix}
\T A_q & 0
\\
0 & 0
\end{matrix}\right],  
\end{equation}
and let us denote $\T N(\V u,\V u):=\T E^{-1}\T G(\V u,\V u)$. Then, equation \eqref{eq:evolution0} (and thus \eqref{eq:Motion_e}) can be equivalently rewritten as 
\begin{equation}\label{eq:evolution}
\frac{\d \V u}{\d t}+\L \V u=\T N(\V u,\V u),\qquad \V u(0)=\V u_0.
\end{equation}

\paragraph{Step 2. Properties of the linear operator $\L$}\cite[Theorem 2]{Abels2009} implies that $\T A_q\in \mathcal{BIP}(X_0)$ with angle $\theta_{\T A_q}=0$ and $0\in\varrho(\T A_q)$. 
 
\noindent Consider the linear operator ${\L}_q:=\T E^{-1}_q\T A_q$ with domain $D({\L}_q)\equiv D(\T A_q)$, and for every $\V u\in \mathcal H_q(\mathcal V)$ 
\begin{equation}\label{eq:E_q}
\T E_q\V u:=\V u+\mathcal P_q\left(b(\V u)\times \V x\right)=\V u+\mathcal P_q\left(\V x\times \T I^{-1}\cdot\int_{\mathcal V}\tilde \rho \V x\times \V u\; \d V\right)
\in \mathcal H_q(\mathcal V). 
\end{equation}
With an argument similar to the one done in {\em Step 1}, it can be shown that 
\[
\L=\left[\begin{matrix} \L_q & \T 0
\\
\T 0 & \T 0\end{matrix}\right]
\]
and ${\L}_q=(\T 1+\T C_q)\T A_q$ with $\T C_q:\;\mathcal H_q(\mathcal V)\to  \mathcal R(\mathcal V)\cap \mathcal H_q(\mathcal V)$ a bounded operator with a finite dimensional range, and then compact. Since $\L_q$ is a compact perturbation of $\T A_q$, $\L_q$ has compact resolvent. In addition, its spectrum consists entirely of eigenvalues of finite algebraic multiplicity, and it is independent of $q$. From Lemma \ref{le:kokr}, it follows that $\L_q$ is positive definite on $\mathcal H_q(\mathcal V)$ when $q=2$. Thus, $\sigma(\L_q)\subset (0,\infty)$. In particular, $0\in \varrho(\L_q)$. 

\noindent The operator $\T B_q:= \T C_q\T A_q$ is bounded from $D(\T A_q)$ to $\mathcal R(\mathcal V)\cap \mathcal H_q(\mathcal V)$. In particular, there exists $s\in (0,1/q)$ such that 
\[
\T B_q:\; D(\T A_q) \to D(\T A^{s/2}_q)\quad\text{ is bounded,}
\]
where $D(\T A_q^{s/2})=[\mathcal H_q(\mathcal V), D(\T A_q)]_{s/2}$ (by \cite[Theorem 3.3.7]{PrSi}). 

\noindent Proposition \ref{prop:perturbation} and Remark \ref{rem:perturbation} imply that $\L_q\in \mathcal{BIP}(\mathcal H_q(\mathcal V)$, and then $\L\in \mathcal{BIP}(\mathcal H_q(\mathcal V)\times \R^3)$ with angle $\theta_{\L_q}<\pi/2$. 

\paragraph{Step 3. The nonlinear term} For $\beta\in (0,1)$, let $X_\beta:=[X_0,X_1]_\beta$.
Then we have
$X_\beta={\mathcal H^{2\beta}_{q}(\mathcal V)}\times \R^3$, where
$\mathcal H^{2\beta}_{q}(\mathcal V)$ is defined by
\begin{equation*}
\mathcal H^{2\beta}_{q}(\mathcal V):=
\left\{ 
\begin{aligned}
&\{\V u\in  H^{s}_{q}(\mathcal V)\cap\mathcal H^q(\mathcal V): \V u=\V 0\; \text{on}\; \partial\mathcal C\}, &&s>1/q,\\
&H^{s}_{q}(\mathcal V)\cap \mathcal H_q(\mathcal V), && s\in [0,1/q).\\
\end{aligned}
\right.
\end{equation*}
The fact that $\T N:=\T E^{-1}\T G:\; X_\beta\times X_\beta\to X_0$ is  bounded for  $\beta=\frac{1}{4}\big(1+\frac{3}{q}\big)$ with $q\in(1,3)$ follows from standard estimates (see e.g. \cite[Section 3]{PrWi} and \cite[proof of Theorem 3.4]{MaPrSi}). For such choice of $\beta$, \eqref{assumptions-pq} implies that $\upmu_{\rm crit}\le 1$. 
\newline

It remains to prove part (b). We note that, under the stated hypotheses,  
\[
{B}^{2\upmu_1 -2/p_1}_{q_1 p_1,\sigma}(\mathcal V)\times \R^3\hookrightarrow
{B}^{2\upmu_2 -2/p_2}_{q_2 p_2,\sigma}(\mathcal V)\times \R^3
\]
and for each fixed $j=1,2$, solutions $\V u_j\equiv(\tilde{\V v_j},\V \omega_{1,j})$ to \eqref{eq:evolution} in the class 
\[
\mathbb{E}_{1,\upmu_j}(0,T):=H^1_{p_j,\upmu_j}((0,T); \mathcal H_{q_j}(\mathcal V)\times\R^3)\cap L^{p_j}_{\upmu_j}((0,T); D(\T A_{q_j})\times\R^3)
\]
are fixed points of the strict contraction 
\[
{\sf T}:\; \mathbb{M}_j\to\mathbb{M}_j,\qquad {\sf T}\V u:=\e^{-t\T L}\V u_0+\e^{-t\T L}*\T N(\V u,\V u),
\]
where $\mathbb{M}_j$ is a closed subset of $\mathbb{E}_{1,\upmu_j}(0,T)$. Since also ${\sf T}:\; \mathbb{M}_1\cap\mathbb{M}_2\to\mathbb{M}_1\cap\mathbb{M}_2$ is a strict contraction, then it admits a unique fixed point which is the unique solution $(\V{\tilde v},\V \omega_1)$ in the class
\begin{equation*}
\begin{split}
&H^1_{p_1,\upmu_1}((0,T); \mathcal H_{q_1}(\mathcal V)\times\R^3)\cap L_{\upmu_1}^{p_1}((0,T); D(\T A_{q_1})\times\R^3) \\
&\cap H^1_{p_2,\upmu_2}((0,T); \mathcal H_{q_2}(\mathcal V)\times\R^3)\cap L_{\upmu_2}^{p_2}((0,T);D(\T A_{q_2})\times\R^3).
\end{split}
\end{equation*}
\end{proof}

\begin{remark}
\begin{itemize}
\item[(a)] In the case $p_1=q_1=2$, we  obtain $\mu_{\rm crit}=3/4$ and 
we find the largest space of initial data $X_{\rm crit}$,
\begin{equation}
\label{p=q=2}
X_{\rm crit}:=(\mathcal H_2(\mathcal V)\times\R^3, \mathcal H^2_2(\mathcal V)\times\R^3)_{1/4,2}
\subset \mathcal {H}^{1/2}_2(\mathcal V)\times \R^3,
\end{equation}
corresponding to which, there exists a unique solution to \eqref{eq:evolution} in the class
\begin{equation*}
\begin{split} 
&H^1_{2,3/4}((0,T); \mathcal H_2(\mathcal V)\times\R^3)\cap L^{2}_{3/4}((0,T); \mathcal H^2_2(\mathcal V)\times\R^3) \\
&\cap H^1_{p,\mu}((0,T); \mathcal H_q(\mathcal V)\times\R^3)\cap L^{p}_{\mu_2}((0,T); \mathcal H^2_q(\mathcal V)\times\R^3),
\end{split}
\end{equation*}
for any $p\ge 2, q\in [2,3)$, with $\mu=1/p + 3/2q -1/2.$
In particular, we can conclude that $v\in C((0,t_+); B^{2-2/p}_{qp}(\mathcal V))$ for any $p\ge 2, q\in [2,3)$.
\item[(b)] Theorem~\ref{th:strong}(b) asserts that problem \eqref{eq:evolution} admits for each initial value
$$(\V{\tilde v}_0,\V{\omega}_{10})\in \mathcal H^1_2(\mathcal V)\times\R^3$$
a unique solution in the class 
\begin{equation*}
\begin{split}
&W^{1,2}((0,T); \mathcal H_2(\mathcal V)\times\R^3)\cap L^2((0,T); \mathcal H^2_2(\mathcal V)\times\R^3) \\
&\cap H^1_{p,\mu}((0,T); \mathcal H_{q}(\mathcal V)\times\R^3)\cap L^{p}_{\mu}((0,T);\mathcal H^2_{q}(\mathcal V)\times\R^3),
\end{split}
\end{equation*}
for any $p\ge 2, q\in [2,3)$, with $\mu=1/p +3/2q-1/4$.
In particular, we can conclude that $v\in C((0,t_+); B^{2-2/p}_{qp}(\mathcal V))$ for any $p\ge 2, q\in [2,3)$.

\end{itemize}
\end{remark}


\appendix
\gdef\thesection{\Alph{section}} 
\makeatletter
\renewcommand\@seccntformat[1]{Appendix \csname the#1\endcsname.\hspace{0.5em}}
\makeatother
\section{Some useful integral equalities}

We recall some elementary integral equalities that have been widely used in the paper.  

Let $B_R$ the open ball in $\R^3$ with radius $R$, centered at the origin of a coordinate system $\{O;\V e_1,\V e_2,\V e_3\}$. The following equalities hold: 
\begin{enumerate}
\item \begin{equation}\label{eq:xtimesomegatimesx}
\int_{B_R}\V x\times (\V \omega\times \V x)=\frac{8\pi R^5}{15}\V \omega\qquad \text{for all }\V \omega \in \R^3. 
\end{equation}
\item \begin{equation}\label{eq:omegatimesxtimesomegatimesx}
\int_{B_R}(\V \omega \times \V x)\times (\V \xi \times \V x)=\frac{4\pi R^5}{15}\V \omega\times \V \xi\qquad \text{for all }\V \omega,\;\V \xi \in \R^3. 
\end{equation}
\end{enumerate}

\section{$L^p$-maximal regularity in time-weighted spaces}\label{sec:maximal_regularity}
In the following, we briefly present the main ideas and results concerning the abstract theory of $L^p$-maximal regularity in time-weighted spaces that has been used in Section \ref{sec:strong}. 

Consider the semilinear parabolic evolution equation on a Banach space $X_0$  
\begin{equation}\label{eq:abstract}
\frac{\d \V u}{\d t}+\L \V u=\T N(\V u,\V u),\quad t\in (0,T),\qquad \V u(0)=\V u_0\in X_0,
\end{equation}
where $T\in (0,\infty]$, $\L:\; X_1\to X_0$ is a linear bounded operator with $X_1$ dense in $X_0$ and $X_1\hookrightarrow X_0$, the operator $\T N:\;X_\beta\times X_\beta\to X$ is bounded and bilinear with $X_\beta:=[X_0,X_1]_\beta$, for some $\beta\in [0,1)$. 

Last two decades have seen a great mathematical effort to answer the following fundamental question regarding \eqref{eq:abstract}, and quasilinear parabolic evolution equations\footnote{These are evolution equations of the form 
$\displaystyle \frac{\d \V u}{\d t}+\L(\V u) \V u=\V F(\V u)$ for suitable (nonlinear) operators $\L$ and $\V F$, see \cite{PrSi}. In the following, we will focus on the particular case of semilinear parabolic equations, since our governing equations can be rewritten in such form. }, in general: {\em what is the {\em critical space}, i.e., the largest  set of initial data that would ensure well-posedness in the time-weighted $L^p$-space}
\begin{equation}\label{eq:time_weighted_class}
\mathbb E_{1,\upmu}(0,T):=H^1_{p,\upmu}((0,T); X_0)
 \cap L^p_{\upmu}((0,T); X_1)?
\end{equation}
The time-weighted spaces in the above equation have been defined in \eqref{eq:time_weigthed_Lp}. The choice of working in time-weighted spaces is twofold. In fact, estimates in time-weighted spaces not only help in lowering the regularity of initial data, they also exploit the parabolic regularization that is usually expected for solutions to  parabolic problems (c.f.  \cite[Introduction]{LePrWi}). More precisely, assume that, corresponding to $u_0\in X_{\gamma,\upmu}:=(X_0,X_1)_{\upmu-1/p,p}$ with $\upmu\in (1/p,1]$, there exists a unique solution to \eqref{eq:abstract} in $\mathbb E_{1,\upmu}(0,T)$. Note that $\mathbb E_{1,\upmu}(0,T)\hookrightarrow C([0,T];X_{\gamma,\upmu})$. However, since,
\[
\mathbb E_{1,\upmu}(\delta,T)\hookrightarrow W^{1,p}((\delta,T); X_0)\cap L^p((\delta,T); X_1)\hookrightarrow C([\delta,T];(X_0,X_1)_{1-1/p,p}), 
\]
for any small $\delta\in (0,T)$, this means that the solution regularizes instantly provided $\upmu<1$. 

Several works have been devoted to this subject and many applications have been presented, we refer the interested reader to \cite{PrSi04,KoPrWi,LePrWi,PrWi} and to the books \cite{amann,PrSi} for a more comprehensive treatment. In \cite{Pruss2018}, it has been also shown that such {\em critical spaces are scaling invariant}, provided the given equation admits a scaling. More recently, parabolic regularization in time-weighted spaces has turned out to be  useful for ascertaining the long-time behaviour of solutions to the equations governing the inertial motion of fluid-filled rigid bodies (see \cite{MaPrSi,MaPrSi19}). 

The existence and uniqueness of solutions to \eqref{eq:abstract} are obtained as fixed points of the map  
\[
{\sf T}:\; \mathbb{M}\to\mathbb{M},\qquad {\sf T}\V u:=\e^{-t\T L}\V u_0+\e^{-t\T L}*\T N(\V u, \V u),
\]
where $\mathbb{M}$ is a closed subset of $\mathbb{E}_{1,\upmu}(0,T)$ (see \cite[proof of Theorem 2.1]{Pruss2018}). Unfortunately, for the above map to be a strict contraction on $\mathbb{E}_{1,\upmu}(0,T)$, it is not enough that $-\L$ is the generator of an analytic semigroup. As a matter of fact, difficulties arise already at the linear level, when one seeks to prove well-posedness of the linear problem   
\begin{equation}\label{eq:abstract_linear}
\frac{\d \V u}{\d t}+\L \V u=\V f,\quad t\in (0,T),\qquad \V u(0)=\V 0\in X_0,
\end{equation}
in the class $W^{1,p}((0,T);X_0)\cap L^p((0,T);X_1)$ for any given $\V f\in L^p((0,T);X_0)$ (so that $\L$ has the so-called {\em property of maximal $L^p$-regularity}, see \cite[Section III.4]{amann}). More conditions are needed on the forcing $\V f$, on the Banach space $X_0$ and on the linear operator $\L$ (see the classical results in \cite{DoreVenni}). Among others, two fundamental requirements are:  
\begin{enumerate}
\item $X_0$ must satisfy the {\em unconditional martingale difference property}, or shortly $X_0$ is a {UMD-space}. This condition is equivalent to require that the Hilbert transform is continuous from $L^p(\R, X_0)$ into $L^p(\R, X_0)$ for $p\in (1,\infty)$. Examples of UMD-spaces are finite-dimensional Banach spaces; Hilbert spaces; Lebesgue spaces $L^p(X,\mu; E)$ for a $\sigma$-finite measure space $(X,\mu)$, a UMD-space $E$ and $p\in (1,\infty)$; Closed subspaces, quotients, duals and finite products of UMD-spaces; Complex interpolation spaces and real interpolation spaces of UMD-spaces (see \cite[Subsections III.4.4 \& III.4.5]{amann}).  
\item $\L\in \mathcal{BIP}(X_0)$, i.e., $\L$ is an operator with {\em bounded imaginary powers}. Here is the definition taken from \cite[Sections 3.2--3.4]{PrSi}: 
\begin{definition}
A sectorial operator $\T A$ is said to admit bounded imaginary powers if the operator $\T A^{is}:\; X_0\to X_0$, defined in the sense of the extended functional calculus for sectorial operators introduced in \cite[Sections 3.2]{PrSi}, is a bounded linear operator for each $s\in \R$, and there exists a constant $C>0$ such that $\norm{\T A^{is}}\le C$ for $|s| \le 1$. The class of such operators is denoted by $\mathcal{BIP}(X_0)$.
\end{definition}
The following representation holds 
\begin{equation}\label{eq:BIP_repr}
\T A^{is}=\frac{1}{2\pi i}\int_\Gamma \lambda^{is}\frac{\lambda}{(1+\lambda)^2}(\T 1+\T A)^2\T A^{-1}(\lambda\T 1-\T A)^{-1}\;\d \lambda,
\end{equation}
where the integration path is taken over $\Gamma=(-\infty, 0]\e^{i\psi}\cup[0,+\infty)\e^{-i\psi}$ with $\phi_{\T A}<\psi<\pi$,\footnote{The integral is independent of $\psi$. It can be replaced by any other curve 
encircling $\sigma(\T A)$ counterclockwise. } $\phi_{\T A}\in [0,\pi)$ sectoriality angle of $\T A$. 

For $\T A\in \mathcal {BIP}(X_0)$, $(\T A^{is})_{s\in \R}$ forms a strongly continuous group of bounded linear operators in $X_0$. The growth bound of this group 
\[
\theta_{\T A}:=\limsup_{|s|\to \infty}\frac{\log\norm{\T A^{is}}}{|s|}
\]
is called {\em power angle }of $\T A$. We refer to \cite[Subsection 3.3.4]{PrSi} (or to \cite[Subsection III.4.7 \& IIII.4.7]{amann}) for other related properties. 
\end{enumerate}
We are now ready to state the main result concerning the local well-posedness of the semilinear evolution equation \eqref{eq:abstract}. 

\begin{theorem}\label{th:MRSP}{\em (\cite[Theorem 2.1 \& Corollary 2.3]{Pruss2018})} 
Let $X_0$ be a UMD-space, $X_1$ be dense in $X_0$ and $X_1\hookrightarrow X_0$. Assume that $\L:\; X_1\to X_0$ is a bounded linear operator and $\L\in \mathcal{BIP}(X_0)$ with power angle $\theta_{\L}<\pi/2$, and $\T N:\;X_\beta\times X_\beta\to X$ is bounded and bilinear with $X_\beta=[X_0,X_1]_\beta\equiv D(\L^\beta)$ for some $\beta\in [0,1)$. 

If $p\in (1,\infty)$, $\upmu\in (1/p,1]$, $\beta\in (\upmu-1/p,1)$ and $2\beta-1\le \upmu-1/p$, then for each $\V u_0\in X_{\gamma,\upmu}$ there exists $T=T(\V u_0)>0$ and a unique solution $\V u$ to \eqref{eq:abstract} in the class 
\[
H^1_{p,\upmu}((0,T); X_0)\cap L^p_{\upmu}((0,T); X_1).
\]

In addition, the solution $\V u$ exists on a maximal time interval $[0,t_+(\V u_0))$, depends continuously on the data and enjoys the additional regularity 
\[
\V u\in W^{1,p}_{loc}((0,t_+);X_0)\cap L^p((0,t_+);X_1)\hookrightarrow C((0,t_+);X_{\gamma,1}). 
\]

Finally, if $\V u([0,t_+))\subset X_{\gamma,\upmu}$ is relatively compact, then the maximal existence time is $t_+=\infty$.
\end{theorem} 

This abstract theory has been widely applied to the Navier-Stokes equations in bounded domains and with different boundary conditions (see e.g. \cite{PrWi,Pruss2018,PrWi18}). Given the structure of our evolution equation \eqref{eq:evolution0} (or equivalently \eqref{eq:evolution}), with the relevant operators defined in \eqref{eq:E} and \eqref{eq:op_L}, the main difficulty in applying Theorem \ref{th:MRSP} resides in proving that our linear operator (given in \eqref{eq:op_L}) satisfies $\L\in \mathcal{BIP}(X_0)$, provided we carefully choose the parameters $p$, $\beta$ and $\upmu$. Proving that an operator admits bounded imaginary powers is not an easy task, in general. For the problem at hand, we will use the following perturbation result. 

\begin{proposition}\label{prop:perturbation}
Let $\T A$ be a sectorial operator on a Banach space $X_0$ 
with sectoriality angle $\phi_{\T A}\in [0,\pi)$, $0\in \varrho(\T A)$, and $\T B:\; D(\T A)\to D(\T A^\alpha)$ be a linear operator, satisfying 
\begin{equation}\label{eq:boundB}
\norm{\T B \V u}_{\T D(\T A^\alpha)}\le c\norm{\V u}_{D(\T A)}\qquad\text{ all }\V u\in D(\T A),
\end{equation}
for some $\alpha\in (0,1]$.  Assume that $0\in \varrho(\T A + \T B)$ and $\T A+\T B$ is sectorial with sectoriality angle $\phi_{\T A+\T B}\in [0,\pi)$. 
If $\T A\in \mathcal{BIP}(X_0)$ with power angle $\theta_{\T A}$, 
then $\T A+\T B\in \mathcal{BIP}(X_0)$ with angle $\theta_{\T A+\T B}\le \max\{\theta_{\T A},\phi_{\T A+\T B}\}$. 
\end{proposition}

In the above statement $D(\T A^\alpha)$ denotes the domain of the (real) fractional power $\T A^\alpha$ of the operator $\T A$.\footnote{Fractional powers are defined for any sectorial operator. } Such domain is endowed with the norm 
\[
\norm{\V u}_{D(\T A^\alpha)}=\norm{\T A^\alpha\V u}_{X_0}\quad \text{for all }\alpha\in [0,1],\ \V u\in D(\T A^\alpha). 
\]
In addition, $D(\T A)\hookrightarrow D(\T A^\alpha)\hookrightarrow D(\T A^0)\equiv X_0$ for $\alpha\in (0,1)$. 

The proof of the proposition is in line with the one of \cite[Theorem 2.3]{ABH01}. The assumption $0\in \varrho(\T A)$ can be dropped. The proof is then obtained by a modification of the proof of \cite[Proposition 3.3.9]{PrSi}. Here, we keep the assumption $0\in \varrho(\T A)$ for simplicity, and present a proof for completeness. 

\begin{proof}
The following representation  
\[
\T A^{is}=\frac{1}{2\pi i}\int_\Gamma \frac{\lambda^{is}}{1+\lambda}(\T 1 +\T A)(\lambda\T 1-\T A)^{-1}\; \d\lambda
\]
follows from \eqref{eq:BIP_repr} and Cauchy Integral Theorem, since $0\in \varrho(\T A)$ and  
\[\begin{split}
&\frac{\lambda}{(1+\lambda)^2}(\T 1+\T A)^2\T A^{-1}(\lambda\T 1-\T A)^{-1}
-\frac{1}{1+\lambda}(\T 1+\T A)(\lambda \T 1-\T A)^{-1}=\frac{1}{(1+\lambda)^2}(\T A^{-1}+\T 1). 
\end{split}\]

Let $\Gamma=(-\infty, 0]\e^{i\psi}\cup[0,+\infty)\e^{-i\psi}$ with $\phi_{\T A+\T B}<\psi<\pi$. We will show that 
\[
\frac{1}{2\pi i}\int_\Gamma \frac{\lambda^{is}}{1+\lambda}(\T 1 +\T A+\T B)(\lambda\T 1-\T A-\T B)^{-1}\; \d\lambda
\]
defines a (linear) bounded operator from $X_0$ to $X_0$, with a uniform bound for $s\in [-1,1]$. Since $\T A\in \mathcal{BIP}(X_0)$, then  
\[
\frac{1}{2\pi i}\int_\Gamma \frac{\lambda^{is}}{1+\lambda}(\T 1 +\T A)(\lambda\T 1-\T A)^{-1}\; \d\lambda=\T A^{is}
\]
and 
\[
\frac{1}{2\pi i}\int_\Gamma \frac{\lambda^{is}}{1+\lambda}\T B(\lambda\T 1-\T A)^{-1}\; \d\lambda=\T B(\T 1 +\T A)^{-1}\T A^{is}
\]
are bounded operators, and they are uniformly bounded for $s\in [-1,1]$. So, it remains to show that the following operator enjoys the same properties 
\begin{equation}\label{eq:ip1}\begin{split}
\frac{1}{2\pi i}\int_\Gamma \frac{\lambda^{is}}{1+\lambda}(\T 1 +\T A+\T B)(\lambda\T 1-\T A-\T B)^{-1}\; \d\lambda-\frac{1}{2\pi i}\int_\Gamma \frac{\lambda^{is}}{1+\lambda}(\T 1 +\T A+\T B)(\lambda\T 1-\T A)^{-1}\; \d\lambda.
\end{split}\end{equation}
Since 
\[
(\lambda\T 1-\T A-\T B)^{-1}-(\lambda\T 1-\T A)^{-1}=(\lambda\T 1-\T A)^{-1}\T B(\lambda\T 1-\T A-\T B)^{-1}, 
\]
it is then enough to show that 
\begin{equation}\label{eq:ip2}
\frac{1}{2\pi i}\int_\Gamma \frac{\lambda^{is}}{1+\lambda}(\T 1 +\T A+\T B)(\lambda\T 1-\T A)^{-1}\T B(\lambda\T 1-\T A-\T B)^{-1}\; \d\lambda
\end{equation}
is bounded. Note that 
\[
\frac{1}{2\pi i}\int_\Gamma \frac{\lambda^{is}}{1+\lambda}(\lambda\T 1-\T A)^{-1}\T B(\lambda\T 1-\T A-\T B)^{-1}\; \d\lambda
\]
is bounded since $\T A$ and $\T A+\T B$ are sectorial operators, $0\in \varrho(\T A)$, $\T B$ satisfies \eqref{eq:boundB}, and then 
\[\begin{split}
(|\lambda|+1)&\norm{(\lambda\T 1-\T A)^{-1}\T B(\lambda\T 1-\T A-\T B)^{-1}\V u}_{X_0}
\\
&\le \left[|\lambda|\norm{(\lambda\T 1-\T A)^{-1}}+\norm{\T A^{-1}}\norm{\T A(\lambda\T 1-\T A)^{-1}}\right]\norm{\T B(\lambda\T 1-\T A-\T B)^{-1}\V u}_{X_0}
\\
&\le k_1\norm{\T B(\lambda\T 1-\T A-\T B)^{-1}\V u}_{X_0}
\le  k_2\norm{(\lambda\T 1-\T A-\T B)^{-1}\V u}_{D(\T A)} \le  k_3\norm{\V u}_{X_0}
\end{split}\]
for all $\V u\in X_0$. Finally, the following estimates  
\[\begin{split}
&\norm{(\T A+\T B)(\lambda\T 1-\T A)^{-1}\T B(\lambda\T 1-\T A-\T B)^{-1}\V u}_{X_0}
\\
&\qquad \qquad \le k_4\norm{(\lambda\T 1-\T A)^{-1}\T B(\lambda\T 1-\T A-\T B)^{-1}\V u}_{D(\T A)}
\\
&\qquad \qquad\le \frac{k_5}{|\lambda|^\alpha}\norm{\T B(\lambda\T 1-\T A-\T B)^{-1}\V u}_{D(\T A^\alpha)}
\le \frac{k_6}{|\lambda|^\alpha}\norm{(\lambda\T 1-\T A-\T B)^{-1}\V u}_{D(\T A)}
\le \frac{k_7}{|\lambda|^\alpha}\norm{\V u}_{X_0}
\end{split}\]
hold for all $\V u\in X_0$ thanks to \eqref{eq:boundB}, the fact that $0\in \varrho(\T A)$, $\T A$ and $\T A+\T B$ are sectorial,  and the estimate (see e.g. \cite[Page 64]{lunardi})
\[
\norm{\T A(\lambda\T 1-\T A)^{-1}\V w}_{X_0}\le \frac{c}{|\lambda|^\alpha}\norm{\V w}_{D(\T A^\alpha)}\quad\text{for all }\V w\in D(\T A^\alpha).
\]
\end{proof}

\begin{remark}\label{rem:perturbation}
\begin{itemize}
\item[{\em (R1)}] With the same assumptions of the above proposition, if we further assume that  ${\sf Re}\;\sigma(\T A)\subset (0,\infty)$, \footnote{ ${\sf Re}\;\sigma(\T A):=\{{\sf Re}\; \lambda:\; \lambda \in \sigma(\T A)\}$. } then the hypothesis $\T A+\T B$ sectorial would be automatically satisfied by \cite[Propositions 2.4.1 \&  2.2.15]{lunardi}. 
\item[{\em (R2)}] Let $X$ and $Y$ be two Banach spaces, and assume that $\dim Y<\infty$. Consider the linear operator 
\[
\T L=\left[\begin{matrix} \T A & \T 0
\\
\T 0 & \T B\end{matrix}\right]:\; \V u=\left[\begin{matrix}\V v \\ \V w\end{matrix}\right]\in D(\T A)\times Y\to \T L\V u=\left[\begin{matrix}\T A\V v \\ \T B\V w\end{matrix}\right]\in X\times Y. 
\] 
The Banach space $X\times Y$ is endowed with the norm $\norm{\cdot}_{X\times Y}:=\norm{\cdot}_X+\norm{\cdot}_Y$. Note that any linear operator from a finite dimensional space to a finite dimensional space admits bounded imaginary powers (as a matter of fact, the path $\Gamma$ in \eqref{eq:BIP_repr} does not need to go to infinity as the spectrum consists of a finite number of eigenvalues, and then the integral representation \eqref{eq:BIP_repr} defines a uniformly bounded operator). Therefore, for $\T A\in \mathcal{BIP}(X)$, we have that $\T L\in\mathcal{BIP}(X\times Y)$.   
\end{itemize}
\end{remark}

\section*{Aknowlegments}
I would like to express my gratitude to Gieri Simonett for the illuminating discussions. 

\bibliographystyle{elsarticle-harv}
\bibliography{Mazzone_Liquid-filled_Gap_R2}   

\end{document}